\documentclass[10pt,reqno]{amsart}
\usepackage{graphicx}
\usepackage{amsmath, amssymb}
\usepackage{color}
\usepackage{url}
\newtheorem{theorem}{Theorem}[section]
\newtheorem{proposition}[theorem]{Proposition} \newtheorem{corollary}[theorem]{Corollary}
\newtheorem{lemma}[theorem]{Lemma}
\newtheorem{remark}[theorem]{Remark}

\newtheorem{definition}[theorem]{Definition}

\begin{document}

\title[Matrix liberation process I]{Matrix liberation process \\ \Small{I: Large deviation upper bound and almost sure convergence}}
\author{Yoshimichi Ueda}
\address{
Graduate School of Mathematics, 
Nagoya University, 
Furocho, Chikusaku, Nagoya, 464-8602, Japan
}
\email{ueda@math.nagoya-u.ac.jp}
\thanks{Supported by Grant-in-Aid for Challenging Exploratory Research 16K13762.}
\date{Feb.\ 15th, 2018}
\begin{abstract}
We introduce the concept of matrix liberation process, a random matrix counterpart of the liberation process in free probability, and prove a large deviation upper bound for its empirical distribution and several properties on its rate function. As a simple consequence we obtain the almost sure convergence of the empirical distribution of the matrix liberation process to that of the corresponding liberation process \emph{as continuous processes} in the large $N$ limit.
\end{abstract}
\maketitle

\allowdisplaybreaks{

\section{Introduction}

Let $M_N(\mathbb{C})^\mathrm{sa}$ be all the $N\times N$ self-adjoint matrices endowed with the natural inner product $\langle A, B\rangle_\mathrm{HS} := \mathrm{Tr}_N(AB)$, and it has the following natural orthogonal basis: 
$$
C_{\alpha \beta} := 
\begin{cases} 
\frac{1}{\sqrt{2}}(E_{\alpha \beta} + E_{\beta \alpha}) & (1 \leq \alpha < \beta \leq N), \\
\qquad\quad E_{\alpha\alpha} & (1 \leq \alpha = \beta \leq N), \\
\frac{\mathrm{i}}{\sqrt{2}}(E_{\alpha\beta}-E_{\beta\alpha}) & (1 \leq \beta < \alpha \leq N). 
\end{cases}
$$    
Here, $\mathrm{Tr}_N$ stands for the non-normalized trace (i.e., $\mathrm{Tr}_N(I_N) = N$ with the identity matrix $I_N$) and the $E_{\alpha\beta}$ are $N\times N$ standard matrix units. Using these inner product and orthogonal basis we identify $M_N(\mathbb{C})^\mathrm{sa}$ with the $N^2$-dimensional Euclidean space $\mathbb{R}^{N^2}$, when we use usual stochastic analysis tools on Euclidean spaces. Choose the $nN^2$-dimensional standard Brownian motion $B_{\alpha\beta}^{(i)}$, $1 \leq \alpha,\beta \leq N$, $1 \leq i \leq n$ with natural filtration $\mathcal{F}_t$, and define 
$$
H_N^{(i)}(t) := 
\sum_{\alpha,\beta=1}^N \frac{B^{(i)}_{\alpha\beta}(t)}{\sqrt{N}}C_{\alpha\beta}, \quad t\geq 0, \quad 1 \leq i \leq n, 
$$
which are called the \emph{$n$ independent $N\times N$ self-adjoint matrix Brownian motions} on $M_N(\mathbb{C})^\mathrm{sa}$. The stochastic differential equation (SDE in short)
$$
\mathrm{d}U^{(i)}_N(t) = \mathrm{i}\,\mathrm{d}H^{(i)}_N(t)\,U^{(i)}_N(t) - \frac{1}{2}U^{(i)}_N(t)\,\mathrm{d}t \quad \text{with} \quad U^{(i)}_N(0) = I_N, \quad 1 \leq i \leq n,  
$$
defines unique $n$ independent diffusion processes $U^{(i)}_N$, $1 \leq i \leq n$, on the $N\times N$ unitary group $\mathrm{U}(N)$, which are called the \emph{$n$ independent $N\times N$ left unitary Brownian motions}. It is known, see e.g., \cite[Lemma 1.4(2)]{Levy:arXiv:1112.2452v2} and its proof, that they satisfy the so-called left increment property, that is, the $U_N^{(i)}(t)U_N^{(i)}(s)^*$, $t \geq s$, are independent of $\mathcal{F}_s$ and has the same distribution as that of $U_N^{(i)}(t-s)$. This property plays a crucial role throughout this article. 

\medskip
For each $1 \leq i \leq n+1$, an $r(i)$-tuple $\Xi_i(N) = (\xi_{ij}(N))_{j=1}^{r(i)}$ of $N\times N$ self-adjoint matrices is given. \emph{Throughout this article, we assume that the given sequence $\Xi(N) := (\Xi_i(N))_{i=1}^{n+1}$ are operator-norm bounded, that is, $\Vert\xi_{ij}(N)\Vert \leq R$ with some constant $R > 0$, and has a limit joint distribution $\sigma_0$ as $N\to\infty$.} See section 2, item {\bf 3} for its precise formulation of $\sigma_0$. Here we introduce the \emph{$N\times N$ matrix liberation process} starting at $\Xi(N)$ as the multi-matrix-valued process 
\begin{align*}
&t \mapsto \Xi^\mathrm{lib}(N)(t) = \big(\Xi_i^\mathrm{lib}(N)(t)\big)_{i=1}^{n+1} = \big((\xi_{ij}^\mathrm{lib}(N)(t))_{j=1}^{r(i)}\big)_{i=1}^{n+1} \\
&\qquad\qquad\qquad\qquad\text{with} 
\quad
\xi_{ij}^\mathrm{lib}(N)(t) := 
\begin{cases} 
U_N^{(i)}(t)\xi_{ij}(N)U_N^{(i)}(t)^* & (1 \leq i \leq n), \\
\qquad \xi_{n+1\,j}(N) & (i=n+1).
\end{cases}
\end{align*}
We emphasize that the matrix liberation process $\Xi^\mathrm{lib}(N)$ is new in random matrix theory and also that each $\Xi_i^\mathrm{lib}(N)$ is a constant process in distribution, that is, its empirical distribution is independent of time, but the whole family $\Xi^\mathrm{lib}(N)$ creates really non-commutative phenomena. 

\medskip
The concept of matrix liberation process comes from the liberation process in free probability defined as follows. Let $(\mathcal{M},\tau)$ be a tracial $W^*$-probability space, and $\mathcal{A}_i \subset \mathcal{M}$, $1 \leq i \leq n+1$, be unital $*$-subalgebras (possibly to be $W^*$-subalgebras). Let $v_i$, $1 \leq i \leq n$, be $n$ freely independent, left free unitary Brownian motions (\cite{Biane:Fields97}) in $(\mathcal{M},\tau)$ with $v_i(0) = 1$, which are ($*$-)freely independent of the $\mathcal{A}_i$. Then the family consisting of $\mathcal{A}_i(t) := v_i(t)\mathcal{A}_i v_i(t)^*$, $1 \leq i \leq n$, and $\mathcal{A}_{n+1}(t) := \mathcal{A}_{n+1}$ converges (in distribution or in moments) to a family of freely independent copies of $\mathcal{A}_i$ as $t \to \infty$. Following Voiculescu \cite{Voiculescu:AdvMath99}, we call this `algebra-valued process' $t \mapsto (\mathcal{A}_i(t))_{i=1}^{n+1}$ the \emph{liberation process starting at $(\mathcal{A}_i)_{i=1}^{n+1}$}. The matrix liberation process $\Xi^\mathrm{lib}(N)$ is a natural random matrix model of the liberation process. The attempt of investigating the matrix liberation process $\Xi^\mathrm{lib}(N)$ is quite natural, because independent large random matrices are typical sources of free independence thanks to the celebrated work of Voiculescu \cite{Voiculescu:InventMath91} on one hand and because, on the other hand, the concept of free independence is central in free probability theory and the liberation process is a `stochastic interpolation' between a given statistical relation and the freely independent one in the free probability framework. 

\medskip
The purpose of this article is to take a first step towards systematic study of the matrix liberation process $\Xi^\mathrm{lib}(N)$ (rather than the unitary Brownian motions $U_N^{(i)}$) with the hope of providing a basis for the study of liberation process and free independence in view of random matrices. Here we take a large deviation phenomenon for its empirical distribution, say $\tau_{\Xi^\mathrm{lib}(N)}$, (see section 2, item {\bf 2} for its formulation) as $N \to \infty$, and actually prove a large deviation upper bound in scale $1/N^2$ as $N \to \infty$. The reader may think that a possible approach is to obtain a large deviation upper bound for the $U_N^{(i)}$ at first and then to use the contraction principle. However, we do not employ such an approach, because we try to find the resulting formula of rate function in as direct a fashion as possible. In fact, the rate function that we will find is constructed by using a certain derivation that is similar to Voiculescu's one in his liberation theory and shown to be good and to have a unique minimizer, which is identified with the empirical distribution $\sigma_0^\mathrm{lib}$ of the liberation process starting at the distribution $\sigma_0$ (see section 2, item {\bf 3} for its precise formulation). Hence the standard Borel--Cantelli argument shows that $\tau_{\Xi^\mathrm{lib}(N)} \rightarrow \sigma_0^\mathrm{lib}$ in the topology of weak convergence uniformly on finite time intervals almost surely as $N\to\infty$. (See the end of the next section for several previously known related results.)

\medskip
Let us take a closer look at the contents of this article. Section 2 is concerned with the framework to capture empirical distributions $\tau_{\Xi^\mathrm{lib}(N)}$ and $\sigma_0^\mathrm{lib}$ in terms of $C^*$-algebras. We emphasize that the $C^*$-algebra language is not avoidable if one wants to discuss the appropriate topology on the space of empirical distributions of non-commutative processes, because $C^*$-algebras are only appropriate, non-commutative counterparts of the spaces of continuous functions over topological spaces. Hence section 2 is just a collection of formulations for several concepts, but important to understand this article. 

We employ the strategy of the celebrated work on independent $N\times N$ self-adjoint Brownian motions due to Biane, Capitaine and Guionnet \cite{BianeCapitaineGuionnet:InventMath03} (also see \cite[part VI, section 18]{Guionnet:LNM}). Namely, we use the exponential martingale of the martingale 
\begin{equation}\label{Eq1}
t \mapsto \mathbb{E}\big[\mathrm{tr}_N(P(\xi_{\bullet\diamond}^\mathrm{lib}(N)(\cdot)))\mid\mathcal{F}_t\big] - \mathbb{E}\big[\mathrm{tr}_N(P(\xi_{\bullet\diamond}^\mathrm{lib}(N)(\cdot)))\big]
\end{equation}
with $\mathrm{tr}_N := \frac{1}{N}\mathrm{Tr}_N$ for any self-adjoint non-commutative polynomial $P$ in indeterminates $x_{ij}(t)$, $1 \leq i \leq n+1$, $1 \leq j \leq r(i)$ and $t \geq 0$, where $P(\xi_{\bullet\diamond}^\mathrm{lib}(N)(\cdot))$ denotes the substitution of $\xi_{ij}^\mathrm{lib}(N)(t)$ for each $x_{ij}(t)$ into the polynomial $P$. Thus we need to compute the resulting exponential martingale by giving the explicit formula of the quadratic variation of the martingale \eqref{Eq1}. This is done in section 3 by utilizing the Clark--Ocone formula in Malliavin calculus. This is similar to \cite{BianeCapitaineGuionnet:InventMath03}, but we need some standard technology on SDEs in the framework of Malliavin calculus (e.g., \cite[chapter 2]{Nualart:Book06}). The key of section 3 is the introduction of a suitable non-commutative derivation, whose formula is not exactly same as but similar to the derivation in Voiculescu's free mutual information \cite{Voiculescu:AdvMath99}. This new derivation will further be investigated elsewhere.  

The resulting quadratic variation involves the conditional expectation with respect to the filtration $\mathcal{F}_t$, and hence we need to investigate its large $N$ limit in the time uniform fashion. This rather technical issue is the theme of section 4, and the proof of the main result there is divided into two parts: We first describe the desired large $N$ limit at each time, and then prove that the convergence is actually uniform in time. In the first part we use the known convergence results on standard Gaussian self-adjoint random matrices, while in the second part the use of Thierry L\'{e}vy's method \cite{Levy:arXiv:1112.2452v2} combining combinatorial techniques with the famous It\^{o} formula is crucial.  

The rest of the discussion goes along a standard strategy in the large deviation theory for hydrodynamics. Namely, we need to prove the exponential tightness of the probability measures in question, and introduce a suitable good rate function by looking at the quadratic variation computed in section 3. These together with proving the large deviation upper bound are done in section 5. In the same section we give a few important properties on the rate function including the fact that $\sigma_0^\mathrm{lib}$ is its unique minimizer, and obtain the almost sure convergence of the empirical distribution $\tau_{\Xi^\mathrm{lib}(N)}$ as continuous processes. The final section 6 is a brief discussion on one of our on-going works in this direction.  

\section{Empirical distributions of (matrix) liberation processes} 

This section is devoted to a natural framework to capture the empirical distributions of (matrix) liberation processes.

\medskip
Let $\mathbb{C}\big\langle x_{\bullet\diamond}(\cdot) \big\rangle := \mathbb{C}\big\langle \{x_{ij}(t)\}_{1 \leq j \leq r(i), 1\leq i \leq n+1, t \geq 0} \big\rangle$ be the universal unital $*$-algebra with subject to $x_{ij}(t)=x_{ij}(t)^*$. We enlarge it to the universal enveloping $C^*$-algebra $C^*_R\big\langle x_{\bullet\diamond}(\cdot)\big\rangle$ with subject to $\Vert x_{ij}(t)\Vert \leq R$. Let $TS\big(C^*_R\big\langle x_{\bullet\diamond}(\cdot)\big\rangle\big)$ be all the tracial states on $C^*_R\big\langle x_{\bullet\diamond}(\cdot)\big\rangle$. We denote by $TS^c\big(C^*_R\big\langle x_{\bullet\diamond}(\cdot)\big\rangle\big)$ the set of $\tau \in TS\big(C^*_R\big\langle x_{\bullet\diamond}(\cdot)\big\rangle\big)$ such that 
$$
t  \mapsto x_{ij}^\tau(t) := \pi_\tau(x_{ij}(t)) \in \pi_\tau\big(\big(C^*_R\big\langle x_{\bullet\diamond}(\cdot)\big\rangle\big)\big) \curvearrowright \mathcal{H}_\tau
$$ 
define strong-operator continuous processes, where $\pi_\tau : C^*_R\big\langle x_{\bullet\diamond}(\cdot)\big\rangle \to B(\mathcal{H}_\tau)$ denotes the GNS representation associated with $\tau$ and the natural lifting of $\tau$ to $\pi_\tau\big(\big(C^*_R\big\langle x_{\bullet\diamond}(\cdot)\big\rangle\big)\big)''$ (the closure in the strong-operator topology) in $ \mathcal{H}_\tau$ is still denoted by the same symbol $\tau$. 

\begin{lemma}\label{L2.1} For any $\tau \in TS\big(C^*_R\big\langle x_{\bullet\diamond}(\cdot)\big\rangle\big)$ the following are equivalent: 
\begin{itemize} 
\item[(1)] $\tau \in TS^c\big(C^*_R\big\langle x_{\bullet\diamond}(\cdot)\big\rangle\big)$. 
\item[(2)] For every $\ell \in \mathbb{N}$ and any possible pairs $(i_1,j_1),\dots,(i_\ell,j_\ell)$ the function 
$$
(t_1,\dots,t_\ell) \in [0,+\infty)^\ell \mapsto 
\tau(x_{i_1 j_1}(t_1)\cdots x_{i_\ell j_\ell}(t_\ell)) \in \mathbb{C}
$$
is continuous. 
\end{itemize}
\end{lemma}
\begin{proof} (1) $\Rightarrow$ (2) is trivial, since $\Vert x_{ij}(t) \Vert \leq R$. 

(2) $\Rightarrow$ (1): For any monomial $P = x_{i_1 j_1}(t_1)\cdots x_{i_\ell j_\ell}(t_\ell)$ one has, by assumption,  
\begin{align*}
&\Vert(x_{ij}^\tau(t) - x_{ij}^\tau(s))\Lambda_\tau(P)\Vert_{\mathcal{H}_\tau}^2 \\
&= 
\tau(x_{i_\ell j_\ell}(t_\ell)\cdots x_{i_1 j_1}(t_1)x_{ij}(t)^2 x_{i_1 j_1}(t_1)\cdots x_{i_\ell j_\ell}(t_\ell)) \\
&\quad-
\tau(x_{i_\ell j_\ell}(t_\ell)\cdots x_{i_1 j_1}(t_1)x_{ij}(t)x_{ij}(s) x_{i_1 j_1}(t_1)\cdots x_{i_\ell j_\ell}(t_\ell)) \\
&\quad-
\tau(x_{i_\ell j_\ell}(t_\ell)\cdots x_{i_1 j_1}(t_1)x_{ij}(s)x_{ij}(t) x_{i_1 j_1}(t_1)\cdots x_{i_\ell j_\ell}(t_\ell)) \\
&\quad+
\tau(x_{i_\ell j_\ell}(t_\ell)\cdots x_{i_1 j_1}(t_1)x_{ij}(s)^2 x_{i_1 j_1}(t_1)\cdots x_{i_\ell j_\ell}(t_\ell)) \\
&\to 0 \quad (\text{as $t \to s$}),  
\end{align*} 
where $\Lambda_\tau : C^*_R\big\langle x_{ij}(\cdot)\big\rangle \to \mathcal{H}_\tau$ denotes the canonical map. Since $\Vert x_{ij}^\tau(t) \Vert \leq \Vert x_{ij}(t) \Vert \leq R$ as above, we conclude that $t \mapsto x_{ij}^\tau(t)$ is strong-operator continuous. 
\end{proof} 

Let $\mathcal{W}_\ell$ be the words of length $\ell$ in indeterminates $x_{ij} = x_{ij}^*$, $1 \leq i \leq n+1$, $1 \leq j \leq r(i)$. For each $w \in \mathcal{W}_\ell$ we denote by $w(t_1,\dots,t_\ell)$ the substitution of $x_{i_k j_k}(t_k)$ for $x_{i_k j_k}$ into $w = x_{i_1 j_1}\cdots x_{i_\ell j_\ell}$. We introduce the function $d : TS^c\big(C^*_R\big\langle x_{\bullet\diamond}(\cdot)\big\rangle\big)\times TS^c\big(C^*_R\big\langle x_{\bullet\diamond}(\cdot)\big\rangle\big) \to [0,+\infty)$ by 
$$
d(\tau_1,\tau_2) := \sum_{m=1}^\infty \sum_{\ell=1}^\infty \frac{1}{2^m(2R)^\ell} \max_{w \in \mathcal{W}_\ell} \sup_{(t_1,\dots,t_\ell) \in [0,m]^\ell}\big| \tau_1(w(t_1,\dots,t_\ell)) - \tau_2(w(t_1,\dots,t_\ell))\big|
$$
for $\tau_1, \tau_2 \in TS^c\big(C^*_R\big\langle x_{\bullet\diamond}(\cdot)\big\rangle\big)$. 

\begin{lemma}\label{L2.2} 
{\rm(1)} $\big(TS^c\big(C^*_R\big\langle x_{\bullet\diamond}(\cdot)\big\rangle\big),d\big)$ is a complete metric space. 

{\rm(2)} For any sequence $(\delta_k)_{k\geq 1}$ of positive real numbers, 
$$
\Gamma_{(\delta_k)} := \bigcap_{k\geq1} \Bigg\{ \tau \in TS^c\big(C^*_R\big\langle x_{\bullet\diamond}(\cdot)\big\rangle\big)\,\Big|\,\sup_{\substack{0 \leq s,t \leq k \\ |s-t| \leq \delta_k}} \max_{\substack{1 \leq j \leq r(i) \\ 1 \leq i \leq n+1}} \tau\big((x_{ij}(s) - x_{ij}(t))^2\big)^{1/2} \leq \frac{1}{k} \Bigg\}  
$$
defines a compact subset in $TS^c\big(C^*_R\big\langle x_{\bullet\diamond}(\cdot)\big\rangle\big)$ endowed with $d$. 
\end{lemma}
\begin{proof} (1) It is easy to see that $d$ defines a metric on $TS^c\big(C^*_R\big\langle x_{\bullet\diamond}(\cdot)\big\rangle\big)$. Thus it suffices to confirm the completeness of the space. 

Let $\tau_p \in TS^c\big(C^*_R\big\langle x_{\bullet\diamond}(\cdot)\big\rangle\big)$ be a Cauchy sequence, that is, $d(\tau_p,\tau_q) \to 0$ as $p,q\to\infty$. For every $w = x_{i_1 j_1}\cdots x_{i_\ell j_\ell} \in \mathcal{W}_\ell$ we have 
$$
|\tau_p(w(t_1,\dots,t_\ell)) - \tau_q(w(t_1,\dots,t_\ell))| \leq 2^m(2R)^\ell d(\tau_p,\tau_q) \to 0 
$$
as $p,q\to \infty$ for every $(t_1,\dots,t_\ell) \in [0,m]^\ell$. Hence, $\lim_{p\to\infty}\tau_p(x_{i_1 j_1}(t_1)\cdots x_{i_1 j_\ell}(t_\ell))$ exists for every word $x_{i_1 j_1}(t_1)\cdots x_{i_\ell j_\ell}(t_\ell)$ in $\mathbb{C}\big\langle x_{\bullet\diamond}(\cdot)\big\rangle$. Since $\mathbb{C}\big\langle  x_{\bullet\diamond}(\cdot)\big\rangle$ is the universal $*$-algebra generated by the $x_{ij}(t) = x_{ij}(t)^*$, the words $x_{i_1 j_1}(t_1)\cdots x_{i_\ell j_\ell}(t_\ell)$ together with the unit $1$ form a linear basis. Hence, we can construct a linear functional $\tau$ on $\mathbb{C}\big\langle x_{\bullet\diamond}(\cdot)\big\rangle$ in such a way that $\tau(1) = 1$ and $\tau(x_{i_1 j_1}(t_1)\cdots x_{i_\ell j_\ell}(t_\ell)) = 
\lim_{p\to\infty}\tau_p(x_{i_1 j_1}(t_1)\cdots x_{i_\ell j_\ell}(t_\ell))$;  
hence $\tau(P) = \lim_{p\to\infty}\tau_p(P)$ for every $P \in \mathbb{C}\big\langle x_{ij}(\cdot)\big\rangle$. Clearly, $\tau$ is a tracial state. We have $|\tau(P)| = \lim_{p\to\infty}|\tau_p(P)| \leq \Vert P \Vert$ for every $P \in \mathbb{C}\big\langle x_{\bullet\diamond}(\cdot)\big\rangle$ ($\hookrightarrow C_R^*\big\langle x_{\bullet\diamond}(\cdot)\big\rangle$ naturally), and therefore, $\tau$ extends a tracial state on $C_R^*\big\langle x_{\bullet\diamond}(\cdot)\big\rangle$. 

Fix $w \in \mathcal{W}_\ell$ and $m \in \mathbb{N}$ for a while. We have 
\begin{align*}
\big|\tau_p(w(t_1,\dots,t_\ell)) - \tau(w(t_1,\dots,t_\ell))\big| 
&= 
\lim_{q \to \infty}\big|\tau_p(w(t_1,\dots,t_\ell)) - \tau_q(w(t_1,\dots,t_\ell))\big| \\
&\leq 
2^m(2R)^\ell \varlimsup_{k'\to\infty} d(\tau_p,\tau_q) 
\end{align*}
for every $(t_1,\dots,t_\ell) \in [0,m]^\ell$; hence 
\begin{align*}
\sup_{(t_1,\dots,t_\ell) \in [0,m]^\ell}\big|\tau_p(w(t_1,\dots,t_\ell)) - \tau(w(t_1,\dots,t_\ell))\big| 
\leq 
2^m(2R)^\ell \varlimsup_{q\to\infty} d(\tau_p,\tau_q). 
\end{align*}
Thus $\tau(w(t_1,\dots,t_\ell)) = \lim_{p\to\infty}\tau_p(w(t_1,\dots,t_\ell))$ is uniform in $(t_1,\dots,t_\ell) \in [0,m]^\ell$. Since $m \in \mathbb{N}$ is arbitrary, we conclude, by Lemma \ref{L2.1}, that $\tau \in TS^c\big(C^*_R\big\langle x_{\bullet\diamond}(\cdot)\big\rangle\big)$. 

(2) Let $\tau_p$ be an arbitrary sequence in $\Gamma_{(\delta_k)}$. For every $m=1,2,\dots$ and every $w \in \mathcal{W}_\ell$, the sequence of continuous functions $\tau_p(w(t_1,\dots,t_\ell))$ is equicontinuous on $[0,m]^\ell$, since  
\begin{align*}
\big|\tau_p(w(t_1,\dots,t_\ell)) - \tau_p(w(t'_1,\dots,t'_\ell))\big| 
\leq 
R^{\ell-1}\sum_{m=1}^\ell \tau_p\big((x_{ij}(t_m) - x_{ij}(t'_m))^2\big)^{1/2}
\end{align*}
by the Cauchy--Schwarz inequality.  Hence, for each $m,\ell=1,2,\dots$, the Arzela-Ascoli theorem (see e.g., \cite[Theorem 11.28]{Rudin:RedBook}) guarantees that any subsequence of $\tau_p$ has a subsequence $\tau_{p'}$ such that $\tau_{p'}(w(t_1,\dots,t_\ell))$ converges uniformly on $[0,m]^\ell$ as $p'\to\infty$ for all $w \in \mathcal{W}_\ell$ ({\it n.b.}\ $\mathcal{W}_\ell$ is a finite set). Then, the usual diagonal argument with respect to $\ell=1,2,\dots$ enables us to select a subsequence $\tau_{p''}$ in such a way that for every $w \in \mathcal{W}_\ell$, $\ell = 1,2,\dots$, the sequence of continuous functions $\tau_{p''}(w(t_1,\dots,t_\ell))$ converges uniformly on $[0,m]^\ell$ for as $p''\to\infty$. This is done for each $m$ and any given subsequence of $\tau_p$. Thus, by the usual diagonal argument again with respect to $m$, we can choose a common subsequence $\tau_{p'''}$ that satisfies the same uniform convergence for all $m$.  In the same way as in the discussion about (1) above we can construct a tracial state $\tau \in TS^c\big(C^*_R\big\langle x_{\bullet\diamond}(\cdot)\big\rangle\big)$ in such a way that $d(\tau_{p'''},\tau) \to 0$ as $p''' \to \infty$. Moreover, for every pair $0 \leq s, t \leq k$ with $|s-t| \leq \delta_k$ and every possible pair $(i,j)$, one has $\tau\big((x_{ij}(s) - x_{ij}(t))^2\big) = \lim_{p''' \to\infty} \tau_{p'''}\big((x_{ij}(s) - x_{ij}(t))^2\big) \leq 1/k^2$, and hence $\tau$ falls into $\Gamma_{(\delta_k)}$. 
\end{proof}

We will provide some notations that will be used throughout the rest of this article. 

\medskip\noindent
{\bf 1. Time-marginal tracial states:} Let $C^*_R\big\langle x_{\bullet\diamond}\big\rangle = C_R^*\big\langle \{x_{ij}\}_{1 \leq i \leq n+1, 1 \leq j \leq r(i)}\big\rangle$ be the universal $C^*$-algebra generated by the $x_{ij} = x_{ij}^*$, $1 \leq i \leq n+1$, $1 \leq j \leq r(i)$ with subject to $\Vert x_{ij} \Vert \leq R$. For each $\mathbf{t} := (t_1,\dots,t_{n+1}) \in [0,+\infty)^{n+1}$, there exists a unique $*$-homomorphism (actually a $*$-isomorphism) $\pi_\mathbf{t} : C^*_R\big\langle x_{\bullet\diamond}\big\rangle \to C^*_R\big\langle x_{\bullet\diamond}(\cdot)\big\rangle$ sending $x_{ij}$ to $x_{ij}(t_i)$. When $t := t_1=\cdots=t_{n+1}$ we simply write $\pi_t := \pi_\mathbf{t}$. The $\pi_\mathbf{t} $ induces a continuous map $\pi_\mathbf{t} ^* : TS^c\big(C^*_R\big\langle x_{\bullet\diamond}(\cdot)\big\rangle\big) \to TS\big(C^*_R\big\langle x_{\bullet\diamond}\big\rangle\big)$ by $\pi_\mathbf{t} ^*(\tau) := \tau\circ\pi_\mathbf{t}$, where $TS\big(C^*_R\big\langle x_{\bullet\diamond}\big\rangle\big)$ is equipped with the $w^*$-topology. By Lemma \ref{L2.1} it is easy to see that $\mathbf{t} \mapsto \pi_\mathbf{t} ^*(\tau)$ is continuous for every $\tau \in TS^c\big(C^*_R\big\langle x_{\bullet\diamond}(\cdot)\big\rangle\big)$. We call $\pi_\mathbf{t}^*(\tau)$ \emph{the marginal tracial state of $\tau$ at multiple time $\mathbf{t}$}. 

\medskip\noindent 
{\bf 2. The empirical distribution $\tau_{\Xi^\mathrm{lib}(N)}$ of $\Xi^\mathrm{lib}(N)$:} The matrix liberation process $\Xi^\mathrm{lib}(N)$ defines $\tau_{\Xi^\mathrm{lib}(N)} \in TS^c\big(C^*_R\big\langle x_{\bullet\diamond}(\,\cdot\,)\big\rangle\big)$ in such a way that 
$$
\tau_{\Xi^\mathrm{lib}(N)}(P) := \mathrm{tr}_N\big(P(\xi_{\bullet\diamond}^\mathrm{lib}(N)(\cdot))\big), \qquad P \in \mathbb{C}\big\langle x_{\bullet\diamond}(\cdot)\big\rangle.
$$
We call this tracial state $\tau_{\Xi^\mathrm{lib}(N)}$ \emph{the empirical distribution of the matrix liberation process $\Xi^\mathrm{lib}(N)$}. The tracial state $\tau_{\Xi^\mathrm{lib}(N)}$ is a random tracial state; actually, it depends upon the $n$ independent left unitary Brownian motions $U_N^{(i)}$ via $\xi^\mathrm{lib}_{ij}(N)$. Hence we have a Borel probability measure $\mathbb{P}(\tau_{\Xi^\mathrm{lib}(N)} \in \,\cdot\,)$ on $TS^c\big(C^*_R\big\langle x_{\bullet\diamond}(\cdot)\big\rangle\big)$, and the large deviation upper bound that we will prove is about the sequence of probability measures $\mathbb{P}(\tau_{\Xi^\mathrm{lib}(N)} \in \,\cdot\,)$. 

\medskip\noindent
{\bf 3. The empirical distribution $\sigma_0^\mathrm{lib}$ of the liberation process with initial distribution $\sigma_0$:} 
The limit joint distribution $\sigma_0$ of the sequence $\Xi(N)$ is defined to be a tracial state on $C^*_R\big\langle x_{\bullet\diamond}\big\rangle$ naturally. Using its GNS construction and taking a suitable free product, we can construct self-adjoint random variables $x^{\sigma_0}_{ij} = x^{\sigma_0}_{ij}{}^*$, $1 \leq i \leq n+1$, $1 \leq j \leq r(i)$ and $n$ freely independent, left free unitary Brownian motions $v_i$, $1 \leq i \leq n$, in a tracial $W^*$-probability space, say $(\mathcal{L},\tilde{\sigma}_0)$, in such a way that the joint distribution of the $x^{\sigma_0}_{ij}$ is indeed $\sigma_0$ and that the $x^{\sigma_0}_{ij}$ and the $v_i$ are freely independent. Thanks to the universality of the $C^*$-algebra $C^*_R\big\langle x_{\bullet\diamond}(\cdot)\big\rangle$, the strong-operator continuous processes 
$$
x_{ij}^{\sigma_0^\mathrm{lib}}(t) := 
\begin{cases} 
 v_i(t)\,x_{ij}^{\sigma_0}\,v_i(t)^* & (1 \leq i \leq n), \\
 \qquad x_{n+1\,j}^{\sigma_0} & (i=n+1) 
 \end{cases} 
$$
define a tracial state $\sigma_0^\mathrm{lib} \in TS^c\big(C^*_R\big\langle x_{\bullet\diamond}(\cdot)\big\rangle\big)$. 

\medskip
Here is a simple fact. 

\begin{proposition}\label{P2.3}  For every $P \in \mathbb{C}\big\langle x_{\bullet\diamond}(\cdot)\big\rangle$ we have $\lim_{N\to\infty}\mathbb{E}\big[\tau_{\Xi^\mathrm{lib}(N)}(P)\big] = \sigma_0^\mathrm{lib}(P)$, that is, $\lim_{N\to\infty}\mathbb{E}\big[\tau_{\Xi^\mathrm{lib}(N)}(\,\cdot\,)\big] = \sigma_0^\mathrm{lib}$ in the weak$^*$-topology.  
\end{proposition} 
\begin{proof} The proof of \cite[Theorem 1(2)]{Biane:Fields97} works well without essential change. 
\end{proof} 

This essentially known fact should be understood as a counterpart of the convergence of finite dimensional distributions, and will be strengthened to the convergence as continuous processes in subsection 5.3. Namely, we will prove that the empirical distribution $\tau_{\Xi^\mathrm{lib}(N)}$ itself converges to $\sigma_0^\mathrm{lib}$ in the metric $d$ almost surely. Here, we briefly mention the known facts concerning the above proposition. The almost-sure version (i.e., without taking the expectation $\mathbb{E}$) of the above proposition has also been known so far (see e.g., the introduction of \cite{CollinsDahlqvistKemp:PTRF1x}); in fact, one can see it in the same way as in \cite[Theorem 1(2)]{Biane:Fields97} with the use of more recent results, for example, \cite[Proposition 6.9]{Levy:AdvMath08} and (the proof of) \cite[Theorem 4.3.5]{HiaiPetz:Book} (see the comment just before Example 4.3.7 there). Moreover, its almost-sure, strong convergence (i.e., the convergence of operator norms) version was recently established by Collins, Dahlqvist and Kemp \cite{CollinsDahlqvistKemp:PTRF1x}. In those results, the event of convergence (whose probability is of course $1$) depends on the choice of time indices $t_1,\dots,t_k$, unlike the fact that we will prove in subsection 5.3.

\section{Computation of Exponential Martingale} 

It is easy to see that, as long as $i \neq n+1$, 
\begin{equation}\label{Eq2}
\begin{aligned}
\big\langle\xi_{ij}^\mathrm{lib}(N)(t),C_{\alpha\beta}\big\rangle_\mathrm{HS} 
&= \big\langle\xi_{ij}(N),C_{\alpha\beta}\big\rangle_\mathrm{HS} \\
&\quad+ \sum_{\alpha',\beta'=1}^N \int_0^t \Big\langle\mathrm{i}\Big[\frac{1}{\sqrt{N}}C_{\alpha'\beta'},\xi_{ij}^\mathrm{lib}(N)(s)\Big],C_{\alpha\beta}\Big\rangle_\mathrm{HS}\,\mathrm{d}B_{\alpha'\beta'}^{(i)}(s) \\
&\quad+ \int_0^t \big\langle\mathrm{tr}_N(\xi_{ij}^\mathrm{lib}(N)(s))I_N -\xi_{ij}^\mathrm{lib}(N)(s),C_{\alpha\beta} \big\rangle_\mathrm{HS}\,\mathrm{d}s
\end{aligned}
\end{equation}
in the Euclidian coordinates on $M_N(\mathbb{C})^\mathrm{sa}$ with respect to the basis $C_{\alpha\beta}$. 

For a given $P = P^* \in \mathbb{C}\big\langle x_{\bullet\diamond}(\cdot)\big\rangle$ the matrix liberation process $t \mapsto \Xi^\mathrm{lib}(N)(t)$ gives the (real-valued) bounded martingale $M_N$ in \eqref{Eq1}, that is, 
$$
M_N(t) = \mathbb{E}\big[\tau_{\Xi^\mathrm{lib}(N)}(P)\mid\mathcal{F}_t\big] - \mathbb{E}\big[\tau_{\Xi^\mathrm{lib}(N)}(P)\big].
$$
The Clark--Ocone formula (see e.g., \cite[Proposition 6.11]{Hu:Book16} for any dimension and \cite[subsection 1.3.4]{Nualart:Book06} for $1$-dimension) asserts that 
$$
M_N(t) 
= \sum_{k=1}^n \sum_{\alpha',\beta'=1}^N \int_0^t \mathbb{E}\big[\mathrm{D}^{(k;\alpha',\beta')}_s\mathrm{tr}_N\big(P(\xi_{\bullet\diamond}^\mathrm{lib}(N)(\cdot))\big)\mid\mathcal{F}_s\big]\,\mathrm{d}B_{\alpha'\beta'}^{(k)}(s), 
$$
where $\mathrm{D}^{(k;\alpha',\beta')}_s$ denotes the Malliavin derivative in the Brownian motion $B_{\alpha'\beta'}^{(k)}$ explained in \cite[p.119]{Nualart:Book06}. The aim of this section is to compute this integrand explicitly by introducing a suitable non-commutative derivative.  

\medskip
Observe that all the coefficients of SDE \eqref{Eq2} are independent of the time parameter and linear in the space variable. Thus, $\mathrm{D}_s^{(k;\alpha',\beta')}\big\langle\xi_{ij}^\mathrm{lib}(N)(t),C_{\alpha\beta}\big\rangle_\mathrm{HS}$ is well-defined. See e.g., \cite[Theorem 2.2.1]{Nualart:Book06} for details. The function $\xi \mapsto U_N^{(i)}(t)\xi U_N^{(i)}(t)^*$ is linear, and hence the matrix-valued process $Y(t)$ in \cite[p.126]{Nualart:Book06} is given by $Y^{(\alpha,\beta)}_{(\alpha',\beta')}(t) = \big\langle U_N^{(i)}(t) C_{\alpha'\beta'} U_N^{(i)}(t)^*,C_{\alpha\beta}\big\rangle_\mathrm{HS}$. By \eqref{Eq2}, the formula \cite[Eq.(2.59)]{Nualart:Book06} enables us to obtain that   
\begin{align*}
&\mathrm{D}_s^{(k;\alpha',\beta')}\big\langle \xi_{ij}^\mathrm{lib}(N)(t),C_{\alpha\beta}\big\rangle_\mathrm{HS} \\
&= 
\delta_{k,i}\mathbf{1}_{[0,t]}(s)\,
\big\langle U_N^{(k)}(t) C_{\alpha_1\beta_1} U_N^{(k)}(t)^*,C_{\alpha\beta}\big\rangle_\mathrm{HS} 
\big\langle U_N^{(k)}(s)^* C_{\alpha_2\beta_2} U_N^{(k)}(s),C_{\alpha_1\beta_1}\big\rangle_\mathrm{HS} \\
&\phantom{aaaaaaaaaaaaaaaaaaaaaaaaaaaaaaaaaaaaa}
\times\Big\langle \mathrm{i}\Big[\frac{1}{\sqrt{N}}C_{\alpha'\beta'},\xi_{ij}^\mathrm{lib}(N)(s)\Big],C_{\alpha_2\beta_2}\Big\rangle_\mathrm{HS} \\
&= 
\frac{\delta_{k,i}\mathbf{1}_{[0,t]}(s)}{\sqrt{N}}\,
\big\langle U_N^{(k)}(t)U_N^{(k)}(s)^*\,\mathrm{i}\,\big[C_{\alpha'\beta'},\xi_{ij}^\mathrm{lib}(N)(s)\big]U_N^{(k)}(s)U_N^{(k)}(t)^*,C_{\alpha\beta}\big\rangle_\mathrm{HS}, 
\end{align*}
where we used the convention of summation over repeated indices $(\alpha_1,\beta_1)$, $(\alpha_2,\beta_2)$ as in \cite[section 2.2]{Nualart:Book06}. For a while, we assume that $P$ is a monomial in the $\xi^\mathrm{lib}_{ij}(N)(t)$. By the Leibniz formula of $\mathrm{D}^{(k;\alpha',\beta')}_s$ we have, for any $\zeta \in \mathbb{C}$, 
\begin{align*}
&\mathrm{D}^{(k;\alpha',\beta')}_s\,\mathrm{Re}\Big(\mathrm{tr}_N\big(\zeta P(\xi_{\bullet\diamond}^\mathrm{lib}(N)(\cdot))\big)\Big) \\
&= 
\frac{1}{\sqrt{N}}\sum_{\substack{P = Q_1 x_{kj}(t) Q_2 \\ s \leq t}} \sum_{\alpha,\beta=1}^N \mathrm{Re}\Big(\zeta\,\mathrm{tr}_N(Q_1 C_{\alpha\beta} Q_2)\Big)\, 
\mathrm{D}_s^{(k;\alpha',\beta')}\big\langle \xi_{kj}^\mathrm{lib}(N)(t),C_{\alpha\beta}\big\rangle_\mathrm{HS} \\
&= 
\frac{1}{\sqrt{N}}\sum_{\substack{P = Q_1 x_{kj}(t) Q_2 \\ s \leq t}} \mathrm{Re}\Big(\zeta\,\mathrm{tr}_N(Q_1 U_N^{(k)}(t)U_N^{(k)}(s)^*\,\mathrm{i}\,\big[C_{\alpha'\beta'},\xi_{kj}^\mathrm{lib}(N)(s)\big]U_N^{(k)}(s)U_N^{(k)}(t)^* Q_2)\Big) \\
&= 
\mathrm{Re}\Bigg(\frac{\zeta\,\mathrm{i}}{\sqrt{N}}\sum_{\substack{P = Q_1 x_{kj}(t) Q_2 \\ s \leq t}} 
\mathrm{tr}_N\big(U_N^{(k)}(s)\big[\xi_{kj}(N),U_N^{(k)}(t)^* Q_2 Q_1\,U_N^{(k)}(t)\big]U_N^{(k)}(s)^*\,C_{\alpha' \beta'}\big)\Bigg), 
\end{align*}
where we identify $Q_l$, $l=1,2$, with $Q_l(\xi_{\bullet\diamond}^\mathrm{lib}(N)(\cdot))$ for short. Here and below we used the convention that the summation $\sum_{P = Q x_{kj}(t) R,\ s \leq t}$ above means that the resulting sum becomes $0$ if no $P = Q x_{kj}(t) R$ with $s \leq t$ occurs.  
Therefore, we conclude that  
\begin{align*} 
&\sum_{k=1}^n \sum_{\alpha',\beta'=1}^N \int_0^t \mathbb{E}\big[\mathrm{D}^{(k;\alpha', \beta')}_s\,\mathrm{Re}\big(\mathrm{tr}_N\big(\zeta P(\xi_{\bullet\diamond}^\mathrm{lib}(N)(\cdot))\big)\big)\mid\mathcal{F}_s\big]\,\mathrm{d}B_{\alpha' \beta'}^{(k)}(s) \\
&=
\sum_{k=1}^n \sum_{\alpha',\beta'=1}^N \int_0^t  \,\frac{\mathrm{d}B_{\alpha' \beta'}^{(k)}(s)}{\sqrt{N}} \times \\
&\quad
\mathrm{Re}\Big(\mathrm{tr}_N\Big(\zeta\,U_N^{(k)}(s)\mathbb{E}\Big[\,\mathrm{i}\,\sum_{\substack{P = Q_1 x_{kj}(t) Q_2 \\ s \leq t}}\big[\xi_{kj}(N),U_N^{(k)}(t)^* Q_2 Q_1\,U_N^{(k)}(t) \big] \mid \mathcal{F}_s\Big]U_N^{(k)}(s)^* C_{\alpha' \beta'}\Big)\Big). 
\end{align*} 
Here, we have used the notation $\mathbb{E}[Y|\mathcal{F}_t] = \big[\mathbb{E}[Y_{ij}|\mathcal{F}_t]\big]$ for a matrix-valued random variable $Y = [Y_{ij}]$, where we naturally extend $\mathbb{E}[\,-\,|\mathcal{F}_t]$ to complex-valued random variables. In the rest of this paper we also write $\mathbb{E}[Y] = \big[\mathbb{E}[Y_{ij}]\big]$. 

\medskip
We are now going back to a general $P = P^* \in \mathbb{C}\big\langle x_{\bullet\diamond}(\cdot)\big\rangle$. Write $P = \sum_l \zeta_l P_l$ with $\zeta_l \in \mathbb{C}$ and monomials $P_l$ in the $\xi^\mathrm{lib}_{ij}(N)(t)$. Then we set
\begin{align*}
&Z_N^{(k)}(s) 
:=  \sum_l \zeta_l\,U_N^{(k)}(s)\,\mathbb{E}\Big[\,\mathrm{i}\sum_{\substack{P_l = Q_{l1} x_{kj}(t) Q_{l2} \\ s \leq t}}\big[\xi_{kj}(N),U_N^{(k)}(t)^* Q_{l2} Q_{l1}\,U_N^{(k)}(t)\big]\mid\mathcal{F}_s\Big]\,U_N^{(k)}(s)^* \\
&= 
\mathbb{E}\Big[\sum_l \zeta_l\,\sum_{\substack{P_l = Q_{l1} x_{kj}(t) Q_{l2} \\ s \leq t}} (U_N^{(k)}(t)U_N^{(k)}(s)^*)^*\,\mathrm{i}\,\big[\xi_{kj}^\mathrm{lib}(N)(t), Q_{l2} Q_{l1} \big](U_N^{(k)}(t)U_N^{(k)}(s)^*)\mid\mathcal{F}_s\Big],
\end{align*}
which can be confirmed to be a self-adjoint matrix valued random variable thanks to $P=P^*$. Since $P=P^*$, that is, $\mathrm{tr}_N\big(P(\xi_{\bullet\diamond}^\mathrm{lib}(N)(\cdot))\big)$ is real-valued, we have 
\begin{align*}
M_N(t) 
&= 
\sum_{k=1}^n \sum_{\alpha',\beta'=1}^N \int_0^t \mathbb{E}\big[\mathrm{D}^{(k;\alpha',\beta')}_s\,\mathrm{Re}\big(\mathrm{tr}_N\big(P(\xi_{\bullet\diamond}^\mathrm{lib}(N)(\cdot))\big)\big)\mid\mathcal{F}_s\big]\,\mathrm{d}B_{\alpha'\beta'}^{(k)}(s) \\
&=
\sum_{k=1}^n \sum_{\alpha',\beta'=1}^N \int_0^t \frac{1}{\sqrt{N}}\mathrm{Re}\big(\mathrm{tr}_N\big(Z_N^{(k)}(s)\,C_{\alpha'\beta'}\big)\big)\,\mathrm{d}B_{\alpha',\beta'}^{(k)}(s) \\
&=
\sum_{k=1}^n \sum_{\alpha',\beta'=1}^N \int_0^t \frac{1}{\sqrt{N}}\mathrm{tr}_N\big(Z_N^{(k)}(s)\,C_{\alpha'\beta'}\big)\,\mathrm{d}B_{\alpha',\beta'}^{(k)}(s)
\end{align*}
and the quadratic variation $\langle M_N \rangle$ of $M_N(t)$ becomes 
\begin{align*} 
\langle M_N \rangle(t) 
&= 
\sum_{k=1}^n \sum_{\alpha',\beta'=1}^N \int_0^t \frac{1}{N}\mathrm{tr}_N\big(Z_N^{(k)}(s) C_{\alpha'\beta'})^2\,\mathrm{d}s \\
&= 
\sum_{k=1}^n \int_0^t \frac{1}{N^3} \sum_{\alpha',\beta'=1}^N \big\langle Z_N^{(k)}(s), C_{\alpha'\beta'} \big\rangle_\mathrm{HS}^2\,\mathrm{d}s \\
&= 
\sum_{k=1}^n \int_0^t \frac{1}{N^3} \big\Vert Z_N^{(k)}(s)\big\Vert_\mathrm{HS}^2\,\mathrm{d}s 
= 
\frac{1}{N^2}\sum_{k=1}^n \int_0^t \Vert Z_N^{(k)}(s)\Vert_{\mathrm{tr}_N,2}^2\,\mathrm{d}s, 
\end{align*}
where we used a well-known formula on stochastic integrals (see e.g., \cite[Proposition 3.2.17, Eq.(3.2.26)]{KaratzasShreve:Book}) as well as $\langle B^{(k)}_{\alpha\beta},B^{(k')}_{\alpha'\beta'}\rangle(t) = \delta_{(k,\alpha,\beta),(k',\alpha',\beta')}\,t$ (see e.g., \cite[Problem 2.5.5]{KaratzasShreve:Book}). 

\bigskip
Here we introduce suitable non-commutative derivations to describe $Z_N^{(k)}(s)$. 

\begin{definition}\label{D3.1} 
We expand $\mathbb{C}\big\langle x_{\bullet\diamond}(\cdot)\big\rangle$ into the universal $*$-algebra 
$$
\mathbb{C}\big\langle x_{\bullet\diamond}(\cdot), v_\bullet(\cdot) \big\rangle := \mathbb{C}\big\langle \{x_{ij}(t)\}_{1 \leq j \leq r(i), 1 \leq i \leq n+1, t \geq 0} \sqcup \{v_i(t)\}_{1 \leq i \leq n, t \geq 0} \big\rangle
$$ 
with subject to $x_{ij}(t) = x_{ij}(t)^*$ and $v_i(t)v_i(t)^* = 1 = v_i(t)^* v_i(t)$, and define the derivations 
$$
\delta_s^{(k)} : \mathbb{C}\big\langle x_{\bullet\diamond}(\cdot)\big\rangle \to \mathbb{C}\big\langle x_{\bullet\diamond}(\cdot), v_\bullet(\cdot) \rangle\otimes_\mathrm{alg}\mathbb{C}\big\langle x_{\bullet\diamond}(\cdot), v_\bullet(\cdot) \big\rangle
$$ 
by 
$$
\delta_s^{(k)} x_{ij}(t) = \delta_{k,i}\mathbf{1}_{[0,t]}(s)\big(x_{kj}(t) v_k(t-s) \otimes v_k(t-s)^* - v_k(t-s) \otimes v_k(t-s)^* x_{kj}(t)\big) 
$$
for $1 \leq k \leq n$. Let $\theta : \mathbb{C}\big\langle x_{\bullet\diamond}(\cdot), v_\bullet(\cdot) \rangle \otimes_\mathrm{alg} \mathbb{C}\langle x_{\bullet\diamond}(\cdot), v_\bullet(\cdot)\big\rangle \to \mathbb{C}\big\langle x_{\bullet\diamond}(\cdot), v_\bullet(\cdot)\big\rangle$ be a linear map defined by $\theta(Q\otimes R) = RQ$, and define 
$$
\mathfrak{D}_s^{(k)} := \theta\circ\delta_s^{(k)} : \mathbb{C}\big\langle x_{\bullet\diamond}(\cdot) \big\rangle \to \mathbb{C}\big\langle x_{\bullet\diamond}(\cdot), v_\bullet(\cdot) \big\rangle
$$
for $1 \leq k \leq n$. 
\end{definition} 

Although it is natural to define $\mathfrak{D}_s^{(k)}$ to be $-\mathrm{i}\,\theta\circ\delta_s^{(k)}$, we drop the scalar multiple $-\mathrm{i}$ in the definition for simplicity. It is easy to confirm that $Z_N^{(k)}(s)$ admits the following formula
$$
Z_N^{(k)}(s) = \mathbb{E}\big[-\mathrm{i}\,(\mathfrak{D}_s^{(k)}P)(\xi_{\bullet\diamond}^\mathrm{lib}(N)(\cdot), U_N^{(\bullet)}(\cdot+s)U_N^{(\bullet)}(s)^*))\mid\mathcal{F}_s\big], 
$$
and hence we have the next proposition thanks to \cite[Corollary 3.5.13]{KaratzasShreve:Book}.  

\begin{proposition}\label{P3.2} For any $P = P^* \in \mathbb{C}\langle x_{ij}(\cdot) \rangle$, we have 
\begin{align*} 
&M_N(t) 
:= \mathbb{E}\big[\tau_{\Xi^\mathrm{lib}(N)}(P)\mid\mathcal{F}_t\big] - \mathbb{E}\big[\tau_{\Xi^\mathrm{lib}(N)}(P)\big] \\
&=
\sum_{k=1}^n \sum_{\alpha',\beta'=1}^N \int_0^t \mathrm{tr}_N\big(\mathbb{E}\big[-\mathrm{i}\,(\mathfrak{D}_s^{(k)}P)(\xi_{\bullet\diamond}^\mathrm{lib}(N)(\cdot), U_N^{(\bullet)}(\cdot+s)U_N^{(\bullet)}(s)^*)\mid\mathcal{F}_s\big]\,C_{\alpha'\beta'}\big)\,\frac{\mathrm{d}B_{\alpha'\beta'}^{(k)}(s)}{\sqrt{N}}, \\
&\big\langle M_N \big\rangle(t) = 
\frac{1}{N^2}\sum_{k=1}^n \int_0^t \big\Vert \mathbb{E}\big[(\mathfrak{D}_s^{(k)}P)(\xi_{\bullet\diamond}^\mathrm{lib}(N)(\cdot)), U_N^{(\bullet)}(\cdot+s)U_N^{(\bullet)}(s)^*)\mid\mathcal{F}_s\big] \big\Vert_{\mathrm{tr}_N,2}^2\,\mathrm{d}s.
\end{align*}
Therefore,
\begin{align*}
&\mathrm{Exp}_N(t) := 
\exp\Big(N^2\Big(\mathbb{E}\big[\tau_{\Xi^\mathrm{lib}(N)}(P)\mid\mathcal{F}_t\big] - \mathbb{E}\big[\tau_{\Xi^\mathrm{lib}(N)}(P)\big] \\
&\qquad\qquad\qquad- \frac{1}{2}\sum_{k=1}^n \int_0^t \big\Vert \mathbb{E}\big[(\mathfrak{D}_s^{(k)}P)(\xi_{\bullet\diamond}^\mathrm{lib}(N)(\cdot), U_N^{(\bullet)}(\cdot+s)U_N^{(\bullet)}(s)^*)\mid\mathcal{F}_s\big] \big\Vert_{\mathrm{tr}_N,2}^2\,\mathrm{d}s\Big)\Big)
\end{align*}
becomes a martingale; hence $\mathbb{E}\big[\mathrm{Exp}_N(t)\big] = \mathbb{E}\big[\mathrm{Exp}_N(0)\big] = 1$. 
\end{proposition} 

For the later use we remark that $-\mathrm{i}\,\mathfrak{D}_s^{(k)}P$ is self-adjoint (since so is $P$), and hence 
\begin{equation}\label{Eq3}
\begin{aligned}
\big\Vert \mathbb{E}\big[(\mathfrak{D}_s^{(k)}P)&(\xi_{\bullet\diamond}^\mathrm{lib}(N)(\cdot)), U_N^{(\bullet)}(\cdot+s)U_N^{(\bullet)}(s)^*)\mid\mathcal{F}_s\big] \big\Vert_{\mathrm{tr}_N,2}^2 \\
&= 
-\mathrm{tr}_N\Big(\mathbb{E}\big[(\mathfrak{D}_s^{(k)}P)(\xi_{\bullet\diamond}^\mathrm{lib}(N)(\cdot)), U_N^{(\bullet)}(\cdot+s)U_N^{(\bullet)}(s)^*)\mid\mathcal{F}_s\big]^2\Big).
\end{aligned}
\end{equation}

\section{Convergence of Conditional Expectation} 

\subsection{Statement} For any given $\tau \in TS^c(C_R^*\big\langle x_{\bullet\diamond}(\cdot)))$ and any $s \geq 0$ we will construct $\tau^s \in  TS^c\big(C_R^*\big\langle x_{\bullet\diamond}(\cdot)\big\rangle\big)$ as follows. Taking a suitable free product, we expand $\big(\pi_\tau\big(C_R^*\big\langle x_{\bullet\diamond}(\cdot)\big\rangle\big),\tau\big)$ to a sufficiently larger tracial $W^*$-probability space, in which we can find $n$ freely independent, left unitary free Brownian motions $v_i^{\tau}$, $1 \leq i \leq n$, that are freely independent of the $x_{ij}^\tau(t)$, $1 \leq j \leq r(i)$, $1 \leq i \leq n+1$, $0 \leq t$ ($\leq s$ if $i \neq n+1$). Then we define new strong-operator continuous processes 
$$
x_{ij}^{\tau^s}(t) := 
\begin{cases} 
v_i^{\tau}((t-s)\vee0)x_{ij}^\tau(t\wedge s)v_i^{\tau}((t-s)\vee0)^* & (1 \leq i \leq n), \\
\qquad\qquad\qquad\ x_{n+1\,j}^\tau(t) & (i=n+1). 
\end{cases}
$$
It is known that there exists a unique $\tau$-preserving conditional expectation $E_s^\tau$ onto the von Neumann subalgebra generated by the $x_{ij}^{\tau^s}(t)$ ($= x_{ij}^\tau(t)$), $1 \leq i \leq n$, $1 \leq j \leq r(i)$, $0 \leq t \leq s$, and the $x_{n+1\,j}^{\tau^s}(t) = x_{n+1\,j}^\tau(t)$, $1 \leq j \leq r(n+1)$, $t \geq 0$,  in the ambient tracial $W^*$-probability space. Via the $*$-homomorphism sending $x_{ij}(t)$  to $x_{ij}^{\tau^s}(t)$, we obtain the desired tracial state $\tau^s \in TS^c\big(C_R^*\big\langle x_{\bullet\diamond}(\cdot)\big\rangle\big)$. 

To each event $\mathcal{E}$, we associate the essential-supremum norm relative to $\mathcal{E}$: 
$$
\Vert X\Vert_{\mathcal{E}} := 
\inf\left\{ L > 0 \mid \mathbb{P}\big(\mathcal{E} \cap \{ |X| > L \}\big) = 0 \right\}
$$ 
for every random variable $X$. Here is the main assertion of this section. 

\begin{theorem}\label{T4.1} For any $\tau \in TS^c(C_R^*\big\langle x_{\bullet\diamond}(\cdot)))$ and $P_1,\dots,P_m \in \mathbb{C}\big\langle x_{\bullet\diamond}(\cdot), v_\bullet(\cdot)\big\rangle$ we have 
\begin{align*} 
\varlimsup_{\varepsilon \searrow 0}\varlimsup_{N\to\infty} 
\sup_{s \geq 0}\Big\Vert \mathrm{tr}_N&\big(
\mathbb{E}\big[P_{1,N}^{(s)}\,\big|\,\mathcal{F}_s\big] \cdots\mathbb{E}\big[P_{m,N}^{(s)}\,\big|\,\mathcal{F}_s\big]\big) - \tau\big(E_s^\tau\big(P^{\tau^s}_1)\cdots E_s^\tau\big(P^{\tau^s}_m\big)\big) 
\Big\Vert_{\mathcal{O}_{\varepsilon}(\tau)} = 0 
\end{align*}
with 
$$ 
P_{k,N}^{(s)} := P_k\big(\xi^\mathrm{lib}_{\bullet\diamond}(N)(\cdot),U_N^{(\bullet)}((\cdot)\vee s)U_N^{(\bullet)}(s)^*\big), \quad P^{\tau^s}_k := P_k\big(x_{\bullet\diamond}^{\tau^s}(\cdot),v_\bullet^{\tau}((\,\cdot\,-s)\vee0)\big)
$$
for $1 \leq k \leq m$ and with $\mathcal{O}_\varepsilon(\tau) := \big\{ d\big(\tau_{\Xi^\mathrm{lib}(N)},\tau\big) < \varepsilon \big\}$, an event. Here we use the same convention such as $\mathbb{E}\big[P_{k,N}\,\big|\,\mathcal{F}_s\big]$ as in section 2. 
\end{theorem}

By definition, $\mathfrak{D}_s^{(k)}P$ with $P \in \mathbb{C}\big\langle x_{\bullet\diamond}(\cdot)\big\rangle$ is a linear combination of monomials of the form $\mathbf{1}_{[0,t]}(s)\,v_k(t-s)^* Q\,v_k(t-s)$ with fixed $Q \in \mathbb{C}\big\langle x_{\bullet\diamond}(\cdot)\big\rangle$ and $t \geq 0$. Hence the next corollary immediately follows from Theorem \ref{T4.1}.   

\begin{corollary}\label{C4.2} For any $\tau \in TS^c\big(C^*\big\langle x_{\bullet\diamond}(\cdot)\big\rangle\big)$ and $P \in \mathbb{C}\big\langle x_{\bullet\diamond}(\cdot)\big\rangle$ we have  
\begin{align*} 
\varlimsup_{\varepsilon \searrow 0}\varlimsup_{N\to\infty} 
\sup_{s \geq 0}\Big\Vert \mathrm{tr}_N\big(
\mathbb{E}\big[(\mathfrak{D}_s^{(k)}P)\big(\xi^\mathrm{lib}_{\bullet\diamond}(N)(\cdot),U_N^{(\bullet)}(\cdot+s)U_N^{(\bullet)}(s)^*\big)\,\big|\,\mathcal{F}_s\big]^2&\big) \\
-\,\tau\big(E_s^\tau\big((\mathfrak{D}_s^{(k)}P)\big(x_{\bullet\diamond}^{\tau^s}(\cdot),v_\bullet^{\tau}(\cdot)\big)\big)^2&\big) \Big\Vert_{\mathcal{O}_{\varepsilon}(\tau)} = 0 
\end{align*}
for every $1 \leq k \leq n$.  
\end{corollary} 

\subsection{Proof of Theorem \ref{T4.1}} 

The proof is divided into two steps; we first prove in subsection 4.2.1 that 
\begin{align*} 
\varlimsup_{\varepsilon \searrow 0}\varlimsup_{N\to\infty} 
\Big\Vert \mathrm{tr}_N&\big(
\mathbb{E}\big[P_{1,N}^{(s)}\,\big|\,\mathcal{F}_s\big] \cdots\mathbb{E}\big[P_{m,N}^{(s)}\,\big|\,\mathcal{F}_s\big]\big) - \tau\big(E_s^\tau\big(P^{\tau^s}_1)\cdots E_s^\tau\big(P^{\tau^s}_m\big)\big) 
\Big\Vert_{\mathcal{O}_{\varepsilon}(\tau)} = 0 
\end{align*}
for each fixed $s \geq 0$, and then in subsection 4.2.2 that the convergence is actually uniform in time $s$. This strategy is motivated by L\'{e}vy's work \cite{Levy:arXiv:1112.2452v2}, and indeed his method is crucial in subsection 4.2.2. A slight generalization of what L\'{e}vy established in \cite{Levy:arXiv:1112.2452v2} is necessary, and thus we will explain it in subsection 4.3 for the reader's convenience. 

Note that all the $P_k$ is `supported' in a finite time interval $[0,T]$, that is, the letters appearing in those $P_k$ are from the $x_{ij}(t)$ and $v_i(t)$ with $t \leq T$. Note also that we may and do assume that all the given $P_k$ are monomials. 

\subsubsection{Convergence at each time $s$}   

Choose another independent $n$-tuple $V_N^{(i)}$ of $N\times N$ left unitary Brownian motions that are independent of the original $n$-tuple $U_N^{(i)}$. Denote by $\mathbb{E}_V$ the expectation only in the stochastic processes $V_N^{(i)}$. 
Define 
$$
\xi_{ij}^\mathrm{lib}(N)^V_s(t) := 
\begin{cases} 
V_N^{(i)}((t-s)\vee0)\,\xi_{ij}^\mathrm{lib}(N)(t \wedge s)\,V_N^{(i)}((t-s)\vee0)^* & (1 \leq i \leq n), \\
\qquad\qquad \xi_{n+1\,j}^\mathrm{lib}(N)(t) = \xi_{n+1\,j}(N) & (i = n+1). 
\end{cases}
$$
Then it is not hard to see that 
\begin{equation*}
\mathbb{E}\big[P_{k,N}^{(s)}\,\big|\,\mathcal{F}_s\big] = \mathbb{E}_V\big[P_k(\xi_{\bullet\diamond}^\mathrm{lib}(N)^V_s(\cdot),V_N^{(\bullet)}((\,\cdot\,-s)\vee0))\big] 
\end{equation*}
due to the left increment property of left unitary Brownian motions. 

Note that $P^V_{k,s,N} := P_k(\xi_{\bullet\diamond}^\mathrm{lib}(N)^V_s(\cdot),V_N^{(\bullet)}((\,\cdot\,-s)\vee0))$ depends only on a finite number of $V_N^{(i)}(t)$ because we have fixed $s$. Each of those $V_N^{(i)}(t)$ is written as 
$$
V_N^{(i)}(t) = W_N^{(i,k)}(t-(k/3))W_N^{(i,k-1)}(1/3)\cdots W_N^{(i,0)}(1/3) \quad \text{or} \quad W_N^{(i,0)}(t)
$$
with $W_N^{(i,l)}(t) := V_N^{(i)}(t+(l/3))V_N^{(i)}(l/3)^*$, $0 \leq t \leq 1/3$. Note that those $W_N^{(i,l)}(t)$ ($0 \leq t \leq 1/3$) become independent, $N\times N$ left unitary Brownian motions. In this way, we may think of $P^V_{k,s,N}$ as a monomial in some $\xi_{ij}^\mathrm{lib}(N)(t)$ (with $t \leq s$ as long as $i \neq n+1$) and some $W_N^{(i,l)}(t),W_N^{(i,l)}(t)^*$ with $0 \leq t \leq 1/3$. Accordingly, we write  $w_{i,l}^{\tau}(t) := v_i^{\tau}(t+(l/3))v_i^{\tau}(l/3)^*$, $0 \leq t \leq 1/3$, $l \in \mathbb{N}$, which become left free unitary Brownian motions. Then $P^{\tau^s}_k = P_k(x_{\bullet\diamond}^{\tau^s}(\cdot),v_\bullet^{\tau}((\,\cdot\,-s)\vee0))$ is also the same monomial as $P^V_{k,s,N}$ with the substitution of $x_{ij}^{\tau}(t)$ and $w_{i,l}^{\tau}(t)$ for $\xi_{ij}^\mathrm{lib}(N)(t)$ and $W_N^{(i,l)}(t)$, respectively. Consequently, it suffices, for the purpose here, to prove that
\begin{equation}\label{Eq4} 
\varlimsup_{\varepsilon \searrow 0}\varlimsup_{N\to\infty} 
\Big\Vert \mathrm{tr}_N\big(\mathbb{E}_W\big[Q_{1,N}\big]\cdots\mathbb{E}_W\big[Q_{m,N}\big]\big) - \tau\big(E_s^\tau(Q^\tau_1)\cdots E_s^\tau(Q^\tau_m)\big) \Big\Vert_{\mathcal{O}_{\varepsilon}(\tau)} = 0 
\end{equation}
with
$$
Q_{k,N} := Q_k(\xi_{\bullet\diamond}^\mathrm{lib}(N)(\cdot),W_N^{(\bullet,\circ)}(\cdot)), \quad 
Q^\tau_k :=  Q_k(x_{\bullet\diamond}^{\tau}(\cdot),w_{\bullet,\circ}^{\tau}(\cdot)), \quad 1 \leq k \leq m
$$
for any given monomials $Q_1,\dots,Q_m$ in indeterminates $x_{ij}(t)$ (with $0 \leq t \leq s$ as long as $i\neq n+1$), $w_{i,l}(t)$, $w_{i,l}(t)^*$ with $0 < t \leq 1/3$, where $\mathbb{E}_W$ denotes the expectation only in the stochastic processes $W_N^{(i,l)}$ and also $Q_{k,N}$ and $Q_k^\tau$ are defined similarly as above. 

\medskip 
Note that the given monomials $Q_1,\dots,Q_m$ depend only on a finite number of indeterminates $x_{i_1 j_1}(t_1),\dots,x_{i_p j_p}(t_p),x_{n+1\,j_{p+1}}(t_{p+1}),\dots,x_{n+1\,j_{p'}}(t_{p'})$ (with $1 \leq i_1,\dots,i_p \leq n$, $0 \leq t_1,\dots,t_p \leq s$) and $w_{i_1 l_1}(t_1),\dots,w_{i_q l_q}(t_q)$ (with $0 < t_1,\dots,t_q \leq 1/3$). As in \cite[section 4]{CollinsDahlqvistKemp:PTRF1x} we may and do write $w_{i_k l_k}^{\tau}(t_k) = f_{t_k}(g_{i_k l_k})$, where $f_{t_k}$ is a continuous function from the real line $\mathbb{R}$ to the $1$-dimensional torus $\mathbb{T}$ (depending only on the time $t_k$) and a standard semicircular system $g_{i_1 l_1},\dots,g_{i_q l_q}$, which is freely independent of $x_{i_1 j_1}^{\tau}(t_1),\dots,x_{i_p j_p}^{\tau}(t_p)$ and $x_{n+1\,j_{p+1}}^{\tau}(t_{p+1}),\dots,x_{n+1\, j_p}^{\tau}(t_p)$. Accordingly, by \cite[Proposition 4.3]{CollinsDahlqvistKemp:PTRF1x} we can choose an independent family of $N\times N$ standard Gaussian self-adjoint random matrices $G_N^{(i_1,l_1)},\dots,G_N^{(i_q,l_q)}$ in such a way that they are independent of the $U_N^{(i_1)}(t_1),\dots,U_N^{(i_p)}(t_p)$ (corresponding to indeterminates $x_{i_1 j_1}(t_1),\dots,x_{i_p j_p}(t_p)$) and the operator norm $\Vert W_N^{(i_k,l_k)}(t_k) - f_{t_k}(G_N^{(i_k,l_k)}) \Vert_{M_N(\mathbb{C})} \to 0$ almost surely as $N \to \infty$. For any $\mathbf{x}, \mathbf{y} \in \mathbb{C}^N$ with $\Vert\mathbf{x}\Vert_{\mathbf{C}^N} \leq 1$, $\Vert\mathbf{y}\Vert_{\mathbb{C}^N} \leq 1$, we have  
\begin{align*}
&\Big|\Big(\Big(\mathbb{E}_W\big[Q_k(\xi_{\bullet\diamond}^\mathrm{lib}(N)(\cdot),W_N^{(\bullet,\circ)}(\cdot))\big] - 
\mathbb{E}_G\big[Q_k(\xi_{\bullet\diamond}^\mathrm{lib}(N)(\cdot),
f_{(\cdot)}(G_N^{(\bullet,\circ)}(\cdot)))\big]\Big)\mathbf{x}\,\Big|\,\mathbf{y}\Big)_{\mathbb{C}^N}\Big| \\
&= 
\Big|\Big(\mathbb{E}_{W\cup G}\big[\xi_{\bullet\diamond}^\mathrm{lib}(N)(\cdot),W_N^{(\bullet,\circ)}(\cdot)) - Q_k(\xi_{\bullet\diamond}^\mathrm{lib}(N)(\cdot),f_{(\cdot)}(G_N^{(\bullet,\circ)}(\cdot)))\big]\mathbf{x}\,\Big|\,\mathbf{y}\Big)_{\mathbb{C}^N}\Big| \\
&\leq 
 \mathbb{E}_{W\cup G}\Big[\big|\big(\big(Q_k(\xi_{\bullet\diamond}^\mathrm{lib}(N)(\cdot),W_N^{(\bullet,\circ)}(\cdot)) - Q_k(\xi_{\bullet\diamond}^\mathrm{lib}(N)(\cdot),f_{(\cdot)}(G_N^{(\bullet,\circ)}(\cdot)))\big)\mathbf{x}\,\big|\,\mathbf{y}\big)_{\mathbb{C}^N}\big|\Big] \\
&\leq 
\mathbb{E}_{W\cup G}\Big[\big\Vert Q_k(\xi_{\bullet\diamond}^\mathrm{lib}(N)(\cdot),W_N^{(\bullet,\circ)}(\cdot)) - Q_k(\xi_{\bullet\diamond}^\mathrm{lib}(N)(\cdot),f_{(\cdot)}(G_N^{(\bullet,\circ)}(\cdot)))\big\Vert_{M_N(\mathbb{C})}\Big]
\end{align*}
with the expectations $\mathbb{E}_G$ and $\mathbb{E}_{W\cup G}$ only in the variables $G_N^{(i,l)}$ and the $W_N^{(i,l)}(t)$, $G_N^{(i,l)}$, respectively. Hence we conclude that 
\begin{align*}
&\Big\Vert \mathbb{E}_W\big[Q_k(\xi_{\bullet\diamond}^\mathrm{lib}(N)(\cdot),W_N^{(\bullet,\circ)}(\cdot))\big] - 
\mathbb{E}_G\big[Q_k(\xi_{\bullet\diamond}^\mathrm{lib}(N)(\cdot),f_{(\cdot)}(G_N^{(\bullet,\circ)}(\cdot)))\big]\Big\Vert_{M_N(\mathbb{C})} \\
&\leq 
 \mathbb{E}_{W\cup G}\Big[\big\Vert Q_k(\xi_{\bullet\diamond}^\mathrm{lib}(N)(\cdot),W_N^{(\bullet,\circ)}(\cdot)) - Q_k(\xi_{\bullet\diamond}^\mathrm{lib}(N)(\cdot),f_{(\cdot)}(G_N^{(\bullet,\circ)}(\cdot)))\big\Vert_{M_N(\mathbb{C})}\Big] \\
&\leq \mathrm{Const.}\,\mathbb{E}\Big[\max_{1\leq k \leq q}\Vert W_N^{(i_k,l_k)}(t_k) - f_{t_k}(G_N^{(i_k,l_k)}) \Vert_{M_N(\mathbb{C})}\Big], 
\end{align*}
since $\big\Vert \xi_{ij}^\mathrm{lib}(N)(t)\big\Vert_{M_N(\mathbb{C})} \leq R$ and since all the $W_N^{(i,l)}(t)$ and $f_t(G_N^{(i,l)}(t))$ are unitary matrices. Since $\max_{1\leq k \leq q}\Vert W_N^{(i_k,l_k)}(t_k) - f_{t_k}(G_N^{(i_k,l_k)}) \Vert_{M_N(\mathbb{C})} \to 0$ almost surely as $N\to\infty$, we conclude that 
\begin{equation}\label{Eq5}
\begin{aligned}
\lim_{N\to\infty}\Big\Vert \mathrm{tr}_N\big(&\mathbb{E}_W\big[Q_{1,N}\big]\cdots\mathbb{E}_W\big[Q_{m,N}\big]\big) \\ 
&- 
\mathrm{tr}_N\big(\mathbb{E}_{W\cup G}\big[Q_{1,N}^{(f_*(G_*))}\big]\cdots\mathbb{E}_{W\cup G}\big[Q_{m,N}^{(f_*(G_*))}\big]\big)\big)\Big\Vert_\infty = 0
\end{aligned}
\end{equation}
with 
$$
Q_{k,N}^{(f_*(G_*))} := Q_k(\xi_{\bullet\diamond}^\mathrm{lib}(N)(\cdot),f_{(\cdot)}(G_N^{(\bullet,\circ)}(\cdot))), 
$$
where $\Vert\,-\,\Vert_\infty$ denotes the essential-supremum norm.  

\medskip
For a given $0 < \delta \leq 1$, the Weierstrass theorem enables us to choose a polynomial $p_{t_k}$ so that the supremum norm $\Vert p_{t_k} - f_{t_k} \Vert_{[-3,3]}$ over the interval $[-3,3]$ is not greater than $\delta$. For a while, we fix such polynomials $p_{t_k}$, $1 \leq k \leq q$. Since $w_{i_k l_k}^{\tau}(t_k) = f_{t_k}(g_{i_k l_k})$ and $\Vert g_{i_k l_k} \Vert \leq 2$, it immediately follows that there exists a positive constant $C > 0$ such that 
\begin{equation}\label{Eq6}
\Big| \tau\big(E_s^\tau(Q^\tau_1)\cdots E_s^\tau(Q^\tau_m)\big) - \tau\big(E_s^\tau\big(Q^{(\tau,p_*(g_*))}_1\big)\cdots E_s^\tau\big(Q^{(\tau,p_*(g_*))}_m\big)\big) \Big| \leq C\delta 
\end{equation}
with 
$$
Q_k^{(\tau,p_*(g_*))} := Q_k\big(x^\tau_{\bullet\diamond}(\cdot),p_{(\cdot)}(g_{\bullet\circ})\big), \quad 1\leq k \leq m. 
$$
Consider the event 
$$
\mathcal{E}_N := 
\bigcap_{k=1}^q \big\{ \big\Vert G_N^{(i_k,l_k)} \big\Vert_{M_N(\mathbb{C})} \leqq 3  \big\},
$$
whose probability $\mathbb{P}(\mathcal{E}_N)$ is known to converge to $1$ as $N\to\infty$ (see e.g., \cite[subsection 5.5]{AndersonGuionnetZeitouni-Book} and references therein). Similarly as above we can find a universal constant $C' > 0$ so that
\begin{equation}\label{Eq7}  
\begin{aligned}
\Big| \mathrm{tr}_N\big(\mathbb{E}_G&\big[\mathbf{1}_{\mathcal{E}_N}\,Q_{1,N}^{(f_*(G_*))} \big]\cdots\mathbb{E}_G\big[\mathbf{1}_{\mathcal{E}_N}\,Q_{m,N}^{(f_*(G_*))} \big]\big) \\
&- \mathrm{tr}_N\big(\mathbb{E}_G\big[\mathbf{1}_{\mathcal{E}_N}\,Q_{1,N}^{(p_*(G_*))} \big]\cdots\mathbb{E}_G\big[\mathbf{1}_{\mathcal{E}_N}\,Q_{m,N}^{(p_*(G_*))} \big]\big) \Big| \leq C'\delta,  
\end{aligned}
\end{equation}
with 
$$
Q_{k,N}^{(p_*(G_*))} := Q_k(\xi_{\bullet\diamond}^\mathrm{lib}(N)(\cdot),p_{(\cdot)}(G_N^{(\bullet,\circ)}(\cdot))),
$$
since $\big\Vert f_{t_k}(G_N^{(i_k,l_k)}) - p_{t_k}(G_N^{(i_k,l_k)})\big\Vert_{M_N(\mathbb{C})} \leq \delta$ on the event $\mathcal{E}_N$. By the `Cauchy--Schwarz inequality' for matricial expectations (see Remark \ref{R4.5} below), we have
\begin{equation}\label{Eq8}  
\begin{aligned}
&\Big\Vert \mathbb{E}_G\big[\mathbf{1}_{\Omega\setminus\mathcal{E}_N}\,Q_{k,N}^{(p_*(G_*))}\big]\Big\Vert_{M_N(\mathbb{C})} \\
&\leq 
\Big\Vert \mathbb{E}_G\big[\mathbf{1}_{\Omega\setminus\mathcal{E}_N}I_N\big] \Big\Vert_{M_N(\mathbf{C})}^{1/2} 
\Big\Vert \mathbb{E}_G\big[(Q_{k,N}^{(p_*(G_*))})^*\,Q_{k,N}^{(p_*(G_*))}\big]\Big\Vert_{M_N(\mathbb{C})}^{1/2} \\
&= 
\big(1-\mathbb{P}(\mathcal{E}_N)\big)^{1/2} 
\Big\Vert \mathbb{E}_G\big[(Q_{k,N}^{(p_*(G_*))})^*\,Q_{k,N}^{(p_*(G_*))}\big]\Big\Vert_{M_N(\mathbb{C})}^{1/2}  
\end{aligned}
\end{equation}
and similarly
\begin{equation}\label{Eq9}
\begin{aligned}
&\Big\Vert\mathbb{E}_G\big[\mathbf{1}_{\Omega\setminus\mathcal{E}_N}\,Q_{k,N}^{(f_*(G_*))}\big]\Big\Vert_{M_N(\mathbb{C})} \\
&\leq  
\big(1-\mathbb{P}(\mathcal{E}_N)\big)^{1/2} 
\Big\Vert \mathbb{E}_G\big[(Q_{k,N}^{(f_*(G_*))})^*\,Q_{k,N}^{(f_*(G_*))}\big]\Big\Vert_{M_N(\mathbb{C})}^{1/2} \\
&\leq  
\big(1-\mathbb{P}(\mathcal{E}_N)\big)^{1/2} 
\mathbb{E}_G\Big[\big\Vert(Q_{k,N}^{(f_*(G_*))})^*\,Q_{k,N}^{(f_*(G_*))}\big\Vert_{M_N(\mathbb{C})}\Big]^{1/2} \\
&\leq 
C_k'' \big(1-\mathbb{P}(\mathcal{E}_N)\big)^{1/2}  
\end{aligned}
\end{equation}
with some constant $C_k'' > 0$ depending only on $Q_k$, since the $f_{t_k}(G_N^{(i_k,l_k)})$ are unitary matrices and $\Vert \xi_{ij}^\mathrm{lib}(N)(t)\Vert_{M_N(\mathbb{C})} \leq R$. Since $\mathbb{P}(\mathcal{E}_N) \to 1$ as $N\to\infty$ as remarked before, we need to prove that 
\begin{equation}\label{Eq10} 
\sup_{N\in\mathbb{N}} \max_{1 \leq k \leq m} \Big\Vert \mathbb{E}_G\big[(Q_{k,N}^{(p_*(G_*))})^*\,Q_{k,N}^{(p_*(G_*))}\big]\Big\Vert_{M_N(\mathbb{C})} < +\infty
\end{equation}
and 
\begin{equation}\label{Eq11} 
\begin{aligned}
\varlimsup_{\varepsilon \searrow 0}\varlimsup_{N\to\infty} 
\Big\Vert \mathrm{tr}_N\big(\mathrm{tr}_N\big(\mathbb{E}_G&\big[Q_{1,N}^{(p_*(G_*))} \big]\cdots\mathbb{E}_G\big[Q_{m,N}^{(p_*(G_*))} \big]\big)\big) \\
&- 
\tau\big(E_s^\tau\big(Q^{(\tau,p_*(g_*))}_1\big)\cdots E_s^\tau\big(Q^{(\tau,p_*(g_*))}_m\big)\big)\Big\Vert_{\mathcal{O}_{\varepsilon}(\tau)} = 0,
\end{aligned}
\end{equation}
both of which are similar to what Biane, Capitaine and Guionnet proved in \cite[section 4]{BianeCapitaineGuionnet:InventMath03}. However, we will give more `exact' proofs to them later for the sake of completeness. In fact, \eqref{Eq8} and \eqref{Eq10} imply that 
\begin{equation}\label{Eq12}
\lim_{N\to\infty}\Big\Vert \big\Vert \mathbb{E}_G\big[\mathbf{1}_{\Omega\setminus\mathcal{E}_N}\,Q_{k,N}^{(p_*(G_*))}\big] \big\Vert_{M_N(\mathbb{C})} \Big\Vert_\infty = 0, \quad 1\leq k \leq m, 
\end{equation}
and moreover, by \eqref{Eq9}
\begin{equation}\label{Eq13}
\lim_{N\to\infty}\Big\Vert \big\Vert \mathbb{E}_G\big[\mathbf{1}_{\Omega\setminus\mathcal{E}_N}\,Q_{k,N}^{(f_*(G_*))}\big] \big\Vert_{M_N(\mathbb{C})} \Big\Vert_\infty = 0, \quad 1\leq k \leq m. 
\end{equation}
Remark that $\big\Vert \Vert \mathbf{1}_{\mathcal{E}_N}\,p_{t_k}(G_N^{(i_k,l_k)})\Vert_{M_N(\mathbb{C})}\big\Vert_\infty \leq \Vert p_{t_k}\Vert_{[-3,3]} \leq 1+ \delta \leq 2$. By \eqref{Eq5}--\eqref{Eq7} and \eqref{Eq11}--\eqref{Eq13}, we have 
\begin{align*} 
\varlimsup_{\varepsilon\searrow 0}\varlimsup_{N\to\infty} 
\Big\Vert \mathrm{tr}_N\big(\mathbb{E}_W\big[Q_{1,N}\big]\cdots\mathbb{E}_W\big[Q_{m,N}\big]\big) - \tau\big(E_s^\tau(Q^\tau_1)\cdots E_s^\tau(Q^\tau_m)\big) \Big\Vert_{\mathcal{O}_{\varepsilon}(\tau)}  
\leq (C' + C)\delta. 
\end{align*}
Hence \eqref{Eq4} follows because $\delta > 0$ can arbitrarily be small and both $C, C'$ are independent of the choice of $\delta > 0$. Hence we have completed the first step expect showing \eqref{Eq10} and \eqref{Eq11}. \qed

\bigskip
Here, we prove \eqref{Eq10} and \eqref{Eq11}. We need two simple lemmas, which are of independent interest because they are very explicit. 

\begin{lemma}\label{L4.3} Let $G^{(i)}$ be an independent sequence of $N \times N$ standard Gaussian self-adjoint random matrices, and $A^{(i)}$ be an sequence of $N \times N$ deterministic matrices. Then we have 
\begin{align*}
\mathbb{E}\big[G^{(p(1))}&A^{(q(1))}\cdots A^{(q(\ell-1))}G^{(p(\ell))}\big] \\
&= 
\sum_{\pi \in \mathcal{P}_2^p(\ell)} N^{\sharp(\gamma_\ell\pi) - 1 - \ell/2}\,\mathrm{tr}_{\pi\gamma_\ell}\big[A^{(q(1))},\dots,A^{(q(\ell-1))},*\,\big].  
\end{align*}
Here, $\mathcal{P}_2^p(\ell)$ is the set of all permutations $\pi$ with $p\circ\pi = p$ whose cycle decompositions consist only of transpositions, $\gamma_\ell$ denotes $(1,2,\dots,\ell)$, $\sharp(\pi\gamma_\ell)$ is the number of cycles in $\pi\gamma_\ell$, and finally $\mathrm{tr}_{\sigma}\big[A^{(q(1))},\dots,A^{(q(\ell-1))},*\big]$, $\sigma \in \mathfrak{S}_\ell$, is defined as follows. If $\sigma$ is a cycle $(i_1,\dots,i_k)$, then it becomes 
$$ 
\begin{cases} 
\mathrm{tr}_N(A^{(q(i_1))}\cdots A^{(q(i_k))})I_N & (i_1,\dots,i_k \leq \ell -1), \\
A^{(q(i_{j+1}))}\,\cdots\,A^{(q(i_k))}A^{(q(i_1))}\,\cdots\,A^{(q(i_{j-1}))} & (i_j = \ell), \\
\qquad\qquad I_N & (k=1, i_1 = \ell),  
\end{cases}
$$
and generally it is to be 
$$
\mathrm{tr}_{\sigma_1}\big[A^{(q(1))},\dots,A^{(q(\ell-1))},*\,\big]\,\cdots\,\mathrm{tr}_{\sigma_m}\big[A^{(q(1))},\dots,A^{(q(\ell-1))},*\,\big]
$$ 
with cycle decomposition $\sigma = \sigma_1\cdots\sigma_m$ (n.b., only one cycle $\sigma_k$ contains $\ell$; hence no ambiguity occurs in the above product because its factors commute with each other). 
\end{lemma}

Note that $\mathrm{tr}_N(\mathrm{tr}_\sigma[A_1,\dots,A_{\ell-1},*\,]\,A_\ell) = \mathrm{tr}_\sigma[A_1,\dots,A_{\ell-1},A_\ell]$ with the notation of \cite[Proposition 22.32]{NicaSpeicher:Book} on the right-hand side. This is the key of the proof below.

\begin{proof} Remark that $\mathrm{tr}_N(XA) = \mathrm{tr}_N\big(\mathbb{E}\big[G^{(p(1))}A^{(q(1))}\cdots A^{(q(\ell-1))}G^{(p(\ell))}\big]A\big)$ for all $A$ forces $X = \mathbb{E}\big[G^{(p(1))}A^{(q(1))}\cdots A^{(q(\ell-1))}G^{(p(\ell))}\big]$. This together with \cite[Proposition 22.32]{NicaSpeicher:Book} (see the above remark) implies the desired result. 
\end{proof} 

This lemma immediately implies \eqref{Eq10}, because $\Vert \xi^\mathrm{lib}_{ij}(N)(t)\Vert_{M_N(\mathbb{C})} \leq R$ and $\sharp(\gamma_{\ell_k}\pi)-1-\ell_k/2 \leq 0$, see e.g., \cite[Exercise 22.15]{NicaSpeicher:Book}. 

\medskip
Similarly as above we derive the next lemma from  \cite[Proposition 22.33]{NicaSpeicher:Book} and its discussion there. 

\begin{lemma}\label{L4.4} Let $(\mathcal{M},\tau)$ be a tracial $W^*$-probability space. Let $g_i$ be an freely independent sequence of standard semicircular elements in $(\mathcal{M},\tau)$, and $a_i$ be an sequence of elements in $\mathcal{M}$ which are freely independent of the $g_i$. Let $E$ be a unique $\tau$-preserving conditional expectation onto the von Neumann subalgebra generated by the $a_i$. Then we have 
\begin{align*}
E\big(g_{p(1)}a_{q(1)}\cdots a_{q(\ell-1)}g_{p(\ell)}\big) 
= 
\sum_{\pi \in NC_2^p(\ell)} \tau_{\pi\gamma_\ell}\big[a_{q(1)},\dots,a_{q(\ell-1)},*\,\big],  
\end{align*}
where $NC_2^p(\ell)$ is the subset of all $\pi \in \mathcal{P}_2^p(\ell)$ that are non-crossing as partitions. The other undefined symbol $\tau_{\pi\gamma_\ell}[a_{q(1)},\dots,a_{q(\ell-1)},*\,]$ is similarly defined as in the previous lemma.
\end{lemma}

It is not so hard to derive \eqref{Eq11} from Lemmas \ref{L4.3}, \ref{L4.4} in the following way: For simplicity we write
\begin{align*}
Q_{k,N}^{(p_*(G_*))} 
&= A_k^{(q(0))}G_k^{(p(1))}A_k^{(q(1))}\cdots A_k^{(q(\ell_k-1))}G_k^{(p(\ell_k))}A_k^{(q(\ell_k))}, \\
Q_k^{(\tau,p_*(g_*))}
&= a_{q(0)}^{(k)}g_{p(1)}^{(k)}a_{q(1)}^{(k)}\cdots a_{q(\ell_k-1)}^{(k)}g_{p(\ell_k)}^{(k)}a_{q(\ell_k)}^{(k)},   
\end{align*}
where each $G_k^{(p(\cdot))}$ (or $A_k^{(q(\cdot))}$) is some of the $G_N^{(i,l)}$ (resp.~a product  of some $\xi_{ij}^\mathrm{lib}(N)(t)$ ($t \leq s$ as long as $i\neq n+1$) or $I_N$) and accordingly, each $g_{p(\cdot)}^{(k)}$ (or $a_{q(\cdot)}^{(k)}$) is some of the $g_{il}$ (resp.~a product of some $x_{ij}^\tau(t)$ ($t \leq s$ as long as $i\neq n+1$) or $1$). Remark that $\sharp(\gamma_{\ell_k}\pi)-1-\ell_k/2$ is always non-positive and equals $0$ if and only if $\pi$ is non-crossing, see e.g., \cite[Exercise 22.15]{NicaSpeicher:Book}. Hence, by Lemmas \ref{L4.3} and \ref{L4.4} we have
\begin{align*}
\mathbb{E}_G\big[Q_{k,N}^{(p_*(G_*))}\big] 
&= 
\sum_{\pi \in NC_2^p(\ell_k)} A_k^{(q(0))}\mathrm{tr}_{\pi\gamma_{\ell_k}}\big[A_k^{(q(1))},\dots,A_k^{(q(\ell_k-1))},*\,\big]\,A_k^{(q(\ell_k))} \\
&\quad+ \sum_{\pi \in \mathcal{P}_2^p(\ell_k)\setminus NC_2^p(\ell_k)} N^{\sharp(\gamma_{\ell_k}\pi)-1-\ell_k/2} \\
&\qquad\qquad\qquad \times A_k^{(q(0))} \mathrm{tr}_{\pi\gamma_{\ell_k}}\big[A_k^{(q(1))},\dots,A_k^{(q(\ell_k-1))},*\,\big]\,A_k^{(q(\ell_k))}, \\
E_s^\tau\big(Q_k^{(\tau,p_*(g_*))}\big) 
&= 
\sum_{\pi \in NC_2^p(\ell_k)} a_{q(0)}^{(k)}\mathrm{tr}_{\pi\gamma_{\ell_k}}\big[a_{q(1)}^{(k)},\dots,a_{q(\ell_k-1)}^{(k)},*\,\big]\,a_{q(\ell_k)}^{(k)}. 
\end{align*}
Therefore, by $\Vert \xi_{ij}^\mathrm{lib}(N)(t) \Vert_{M_N(\mathbb{C})} \leq R$, we obtain that 
\begin{align*}
\Big|\mathrm{tr}_N\big(\mathbb{E}_G\big[Q_{1,N}^{(p_*(G_*))}\big]\cdots \mathbb{E}_G\big[Q_{m,N}^{(p_*(G_*))}\big]\big) - \tau\big(E_s^\tau\big(Q_1^{(\tau,p_*(g_*))}\big)\cdots E_s^\tau\big(Q_m^{(\tau,p_*(g_*))}\big)\big)\Big|& \\
\leq   
\big|S(\tau_{\Xi^\mathrm{lib}(N)}(W_1),\dots,\tau_{\Xi^\mathrm{lib}(N)}(W_L)) - S(\tau(W_1),\dots,\tau(W_L))\big|& + \frac{C'''}{N}
\end{align*}  
with some monomials $W_1,\dots,W_L$ in the $x_{ij}(t)$ and a positive constant $C''' > 0$ (which is independent of $N$), where $S$ is a certain polynomial of commutative indeterminates. 
It follows that 
\begin{align*}
\varlimsup_{N\to\infty}\Big\Vert\mathrm{tr}_N\big(\mathbb{E}_G\big[Q_{1,N}^{(p_*(G_*))}\big]\cdots \mathbb{E}_G\big[Q_{m,N}^{(p_*(G_*))}\big]\big) - \tau\big(E_s^\tau\big(Q_1^{(\tau,p_*(g_*))}\big)\cdots E_s^\tau\big(Q_m^{(\tau,p_*(g_*))}\big)\big)\Big\Vert_{\mathcal{O}_\varepsilon(\tau)}& \\
\leq  
\Big\Vert S(\tau_{\Xi^\mathrm{lib}(N)}(W_1),\dots,\tau_{\Xi^\mathrm{lib}(N)}(W_L)) - S(\tau(W_1),\dots,\tau(W_L))\Big\Vert_{\mathcal{O}_\varepsilon(\tau)}&.
\end{align*} 
By definition, for a given $\delta > 0$, there exists $\varepsilon > 0$ so that for every $0  < \varepsilon' \leq \varepsilon$ one has
$$
\Big\Vert S(\tau_{\Xi^\mathrm{lib}(N)}(W_1),\dots,\tau_{\Xi^\mathrm{lib}(N)}(W_L)) - S(\tau(W_1),\dots,\tau(W_L))\Big\Vert_{\mathcal{O}_{\varepsilon'}(\tau)} \leq \delta. 
$$
Hence we are done. \qed

\begin{remark}\label{R4.5} {\rm (See e.g., \cite[Exercise 3.4 in p.40]{Paulsen:Book}) Let $X = [X_{ij}]$ be a matrix whose entries are integrable. If $X$ is positive-definite almost surely, then so is $\mathbb{E}[X] = [\mathbb{E}[X_{ij}]]$ since $\sum_{i,j} \bar{\zeta}_i\mathbb{E}(X_{ij})\zeta_j = \mathbb{E}\Big[\sum_{i,j} \bar{\zeta}_i X_{ij} \zeta_j\Big] \geq 0$ for any scalars $\zeta_i$. 

Let $A = [A_{ij}]$ and $B = [B_{ij}]$ be $N\times N$ matrices whose entries have all moments. Since 
$$
\begin{bmatrix} \mathbb{E}[A^* A] & \mathbb{E}[A^* B] \\ \mathbb{E}[B^* A] & \mathbb{E}[B^* B] \end{bmatrix} = \mathbb{E}\left[ \begin{bmatrix} A & B \\ O & O \end{bmatrix}^* \begin{bmatrix} A & B \\ O & O \end{bmatrix}\right] \geq O, 
$$
one has, for all $t, \theta \in \mathbb{R}$ and $\mathbf{x}, \mathbf{y} \in \mathbb{C}^N$, 
\begin{align*}
t^2(\mathbb{E}[A^* A]\mathbf{x}|\mathbf{x})_{\mathbb{C}^N} + &2t\,\mathrm{Re}\big(\mathrm{e}^{-\mathrm{i}\theta}(\mathbb{E}[B^* A]\mathbf{x}|\mathbf{y})_{\mathbb{C}^N}\big) + 
(\mathbb{E}[B^* B]\mathbf{y}|\mathbf{y})_{\mathbb{C}^N} \\
&= 
\left(
\begin{bmatrix} \mathbb{E}[A^* A] & \mathbb{E}[A^* B] \\ \mathbb{E}[B^* A] & \mathbb{E}[B^* B] \end{bmatrix}
\begin{bmatrix} t\mathbf{x} \\ \mathrm{e}^{\mathrm{i}\theta}\mathbf{y} \end{bmatrix} 
\Bigg|
\begin{bmatrix} t\mathbf{x} \\ \mathrm{e}^{\mathrm{i}\theta}\mathbf{y} \end{bmatrix}
\right)_{\mathbb{C}^{2N}} \geq 0,
\end{align*}
and hence $\big|(\mathbb{E}[B^* A]\mathbf{x}|\mathbf{y})_{\mathbb{C}^N}\big|^2 \leq 
(\mathbb{E}[A^* A]\mathbf{x}|\mathbf{x})_{\mathbb{C}^N} 
(\mathbb{E}[B^* B]\mathbf{y}|\mathbf{y})_{\mathbb{C}^N}$. It follows that 
$$
\Vert \mathbb{E}[B^* A]\Vert_{M_N(\mathbb{C})}^2 \leq \Vert \mathbb{E}[A^* A]\Vert_{M_N(\mathbb{C})}\,\Vert \mathbb{E}[B^* B]\Vert_{M_N(\mathbb{C})}.
$$}
\end{remark} 

\subsubsection{The convergence is uniform in time $s$} 
Let us introduce the map $\Pi_s : \mathbb{C}\big\langle x_{\bullet\diamond}(\cdot),v_\bullet(\cdot)\big\rangle \to \mathbb{C}\big\langle x_{\bullet\diamond}(\cdot), v_\bullet(\cdot)\big\rangle$ defined by replacing $x_{ij}(t)$ with $v_i((t-s)\vee0)x_{ij}(t\wedge s)v_i((t-s)\vee0)^*$ as long as $i \neq n+1$ and also replacing $v_i(t)$ with $v_i((t-s)\vee0)$, with keeping the other letters. Remark that the resulting $\Pi_s P$ is a (non-commutative) polynomial in the $x_{ij}(t)$ (with $t \leq s$ if $i \neq n+1$) and the $v_i(t)$. 

\medskip
For a while, \emph{we are dealing with an arbitrarily fixed monomial $P$ whose letters are supported in $[0,T]$, that is, the letters are from the $x_{ij}(t)$ and the $v_i(t)$ with $t \leq T$, and so is $\Pi_s P$}. As before we have
$$
\mathbb{E}\big[P(\xi_{\bullet\diamond}^\mathrm{lib}(N)(\cdot),U_N^{(\bullet)}((\cdot)\vee s)U_N^{(\bullet)}(s)^*)\mid\mathcal{F}_s\big] 
= 
\mathbb{E}_V\big[(\Pi_s P)(\xi_{\bullet\diamond}^\mathrm{lib}(N)(\cdot),V_N^{(\bullet)}(\cdot))\big],  
$$
where $V_N^{(i)}$, $1 \leq i \leq n$, are $n$ independent left unitary Brownian motions that are independent of the $U_N^{(i)}(t)$ with $t \leq s$. Denote by $L(P)$ the number of letters in the given monomial $P$ (we call it the length of $P$). Observe that $L(\Pi_s P) \leq 3L(P)$. 

\medskip
In what follows, we fix $s$, but will give our desired estimate in such a way that it is independent of the choice of $s$.

\medskip
Let us introduce the following algorithm: If $v_i(t_1), v_i(t_2),\dots, v_i(t_{\ell_i})$ with $t_1 < t_2 < \cdots < t_{\ell_i}$ ({\it n.b.}, $\ell_i \leq \sum_{i=1}^n \ell_i \leq L(\Pi_s P) \leq 3L(P)$) are all the $v_i(\cdot)$ letters appearing in $\Pi_s P$, we replace these with 
$$
w_{i1},\, w_{i2}w_{i1},\, w_{i3}w_{i2}w_{i1}, \dots,\, w_{i\ell_i}\cdots w_{i2}w_{i1}
$$
with new indeterminates $w_{ij}$ ($1 \leq i \leq n$, $1 \leq j \leq \ell_i$ ($\leq 3L(P)$)), and set $t_{ij} := t_j - t_{j-1}$ with $t_0 := 0$. Applying this algorithm to the monomial $\Pi_s P$, we get a new monomial $\widehat{\Pi}_s P$ whose letters are in the $x_{ij}(t)$ ($0 \leq t \leq s$) and the $w_{ij}$. Observe the following rather rough estimates 
\begin{equation}\label{Eq14}
L(\widehat{\Pi}_s P) \leq L(\Pi_s P)^2 \leq 9L(P)^2, \qquad 
t_{ij} \leq T. 
\end{equation}
Let $W_N^{(i,j)}$ be independent left unitary Brownian motions that are independent of the $U_N^{(i)}(t)$ with $t \leq s$ and denote by $\mathbb{E}_W$ the expectation only in the stochastic processes $W_N^{(i,j)}$. By the left increment property of left unitary Brownian motions we have
\begin{align*}
\mathbb{E}\big[P(\xi_{\bullet\diamond}^\mathrm{lib}(N)(\cdot),U_N^{(\bullet)}((\cdot)\vee s)U_N^{(\bullet)}(s)^*)\mid\mathcal{F}_s\big] 
&=
\mathbb{E}_V\big[(\Pi_s P)(\xi_{\bullet\diamond}^\mathrm{lib}(N)(\cdot),V_N^{(\bullet)}(\cdot))\big] \\
&=
\mathbb{E}_W\big[(\widehat{\Pi}_s P)(\xi_{\bullet\diamond}^\mathrm{lib}(N)(\cdot),W_N^{(\bullet,\circ)}(t_{\bullet\circ}))\big],
\end{align*}
where $(\widehat{\Pi}_s P)(\xi_{\bullet\diamond}^\mathrm{lib}(N)(\cdot),W_N^{(\bullet,\circ)}(t_{\bullet\circ}))$ denotes the substitution of $\xi_{ij}^\mathrm{lib}(t)$ and $W_N^{(i,j)}(t_{ij})$ for $x_{ij}(t)$ and $w_{ij}$, respectively, into $\widehat{\Pi}_s P$.  

\medskip
For simplicity, let us denote $X := \widehat{\Pi}_s P$, and write $X = X(1)\cdots X(\ell)$ with $\ell := L(X)$ whose letters $X(k)$ are from $\{x_{ij}(t)\mid 1 \leq i \leq n, 1 \leq j \leq r(i), 0 \leq t\,(\text{$\leq s$ as long as $i\neq n+1$})\} \cup \{w_{ij},w_{ij}^* \mid 1 \leq i \leq n, 1 \leq j \leq 3L(P)\}$. The substitution of $\xi_{ij}^\mathrm{lib}(N)(t)$ and $W_N^{(i,j)}(t_{ij})$ for $x_{ij}(t)$ and $w_{ij}$, respectively, into the monomial $X$ is denoted by $X_N = X_N(1)\cdots X_N(\ell)$. 

\medskip
Let $X_N(\ell+1) := A \in M_N(\mathbb{C})$ be arbitrarily chosen. Let $\rho : \mathbb{C}[\mathfrak{S}_{l+1}] \curvearrowright (\mathbb{C}^N)^{\otimes(\ell+1)}$,  on which $M_N(\mathbb{C})^{\otimes(\ell+1)}$ acts naturally, be the permutation representation of $\mathfrak{S}_{\ell+1}$ over the tensor product components; in fact, $\rho(\sigma)(e_1\otimes\cdots\otimes e_{\ell+1}) = e_{\sigma^{-1}(1)}\otimes\cdots\otimes e_{\sigma^{-1}(\ell+1)}$ for $\sigma \in \mathfrak{S}_{\ell+1}$. For each $\sigma \in \mathfrak{S}_{\ell+1}$ we define, following \cite[section 3]{Levy:arXiv:1112.2452v2} (rather than Lemma \ref{L4.3}), 
\begin{align*}
\mathrm{tr}_\sigma&(X_N(1),\dots,X_N(\ell),X_N(\ell+1)) \\
:&= 
\frac{1}{N^{\sharp(\sigma)}}\mathrm{Tr}_N^{\otimes(\ell+1)}(\rho(\sigma)\,(X_N(1)\otimes\cdots X_N(\ell) \otimes X_N(\ell+1))) \\
&= 
\prod_{(k_1,\dots,k_*) \preccurlyeq \sigma} \mathrm{tr}_N(X_N(k_*) \cdots X_N(k_1)), 
\end{align*}
where $(k_1,\dots,k_*) \preccurlyeq \sigma$ means that $(k_1,\dots,k_*)$ is a cycle component of the cycle decomposition of $\sigma \in \mathfrak{S}_{\ell+1}$, and $\sharp(\sigma)$ denotes the number of cycles in $\sigma$ as in the previous subsection. \emph{Note that $\mathrm{tr}_\sigma(\cdots)$ here is not consistent with $\mathrm{tr}_\sigma[\cdots]$ in Lemma \ref{L4.3}, but $\mathrm{tr}_{\sigma^{-1}}(\cdots) = \mathrm{tr}_\sigma[\cdots]$ holds.} In particular, for the cycle $\gamma_{\ell+1} = (1,\dots,\ell,\ell+1)$ we have
\begin{align*}
\mathbb{E}_W\big[\mathrm{tr}_{\gamma_{\ell+1}^{-1}}(X_N(1),\dots,X_N(\ell),A)\big] 
= 
\mathrm{tr}_N\big(\mathbb{E}_W\big[X_N\big]A\big). 
\end{align*}
Then by a slight generalization of \cite[Proposition 3.5]{Levy:arXiv:1112.2452v2} (see the next subsection for its precise statement with a detailed proof) there exist universal coefficients $c_\sigma$, $\sigma \in \mathfrak{S}_{\ell+1}$, depending on the $t_{ij}$ and $X$, and a  universal constant $C > 1$, depending only on $T$ and $L(P)$ due to \eqref{Eq14} (and hence only on $P$), such that 
$$
\left|\mathrm{tr}_N\big(\mathbb{E}_W\big[X_N\big]A\big)-\sum_{\sigma \in \mathfrak{S}_{\ell+1}} c_\sigma\,\mathrm{tr}_\sigma\big(X_N^{(0)}(1),\dots,X_N^{(0)}(\ell),A\big)\right| \leq 
\frac{C}{N^2}\,\Vert A \Vert_{\mathrm{tr}_N,1}
$$
and
\begin{equation}\label{Eq15} 
\big|c_\sigma\big| \leq C, \quad \sigma \in \mathfrak{S}_{\ell+1}
\end{equation}
with 
$$
X_N^{(0)}(k) := 
\begin{cases}
I_N & (X(k) = \text{$w_{ij}$ or $w_{ij}^*$}), \\
X_N(k) & (\text{otherwise}),
\end{cases}
$$
since
$\big|\mathrm{tr}_\sigma\big(X_N^{(0)}(1),\dots,X_N^{(0)}(\ell),A\big)\big| \leq (R\vee1)^{L(P)}\,\Vert A \Vert_{1,\mathrm{tr}_N}$ (n.b., the procedure from $P$ to $X[P] = \widehat{\Pi}_s P$ does not make the number of $x_{ij}(t)$ increase), where $\Vert\,-\,\Vert_{1,\mathrm{tr}_N}$ denotes the trace norm with respect to the normalized trace $\mathrm{tr}_N$.  
Since 
$$
\sum_{\sigma \in \mathfrak{S}_{\ell+1}} c_\sigma\,\mathrm{tr}_\sigma\big(X_N^{(0)}(1),\dots,X_N^{(0)}(\ell),A\big) = 
\mathrm{tr}_N\Bigg(\sum_{\sigma \in \mathfrak{S}_{\ell+1}} c_\sigma\,\mathrm{tr}_{\sigma^{-1}}\big[X_N^{(0)}(1),\dots,X_N^{(0)}(\ell),*\,\big]\,A\Bigg) 
$$
with the notation in Lemma \ref{L4.3} and since $A \in M_N(\mathbb{C})$ is arbitrary, we conclude that
\begin{equation}\label{Eq16}
\Big\Vert \mathbb{E}_W\big[X_N\big] - \sum_{\sigma \in \mathfrak{S}_{\ell+1}} c_\sigma\,\mathrm{tr}_{\sigma^{-1}}\big[X_N^{(0)}(1),\dots,X_N^{(0)}(\ell),*\,\big] \Big\Vert_{M_N(\mathbb{C})} \leq \frac{C}{N^2}. 
\end{equation}
Notice that $\mathrm{tr}_{\sigma^{-1}}\big[X_N^{(0)}(1),\dots,X_N^{(0)}(\ell),*\,\big]$ depends only on the traces $\mathrm{tr}_N$ of monomials in the $\xi_{ij}^\mathrm{lib}(N)(t)$ (with $0 \leq t \leq s$ as long as $i \neq n+1$), or other words, the $\tau_{\Xi^\mathrm{lib}(N)}$ of monomials in the $x_{ij}(t)$ (with $0 \leq t \leq s$ as long as $i \neq n+1$). 

\medskip
Observe that \eqref{Eq16} holds for any monomial $P$ and $s \geq 0$, and we should write $X = X[P] := \widehat{\Pi}_s P$, $\ell = \ell_P$ ($= L(X[P])$), $c_\sigma = c_\sigma(P)$ and $C = C_P$ for clarifying the dependency in what follows.    
Set
$$
X[P]^{(0)}(k) := 
\begin{cases}
1 & (X[P](k) = \text{$w_{ij}$ or $w_{ij}^*$}), \\
X[P](k) & (\text{otherwise})
\end{cases}
$$
and for simplicity we write 
\begin{align*} 
\tau_{\Xi^\mathrm{lib}(N)}(\sigma;P) &:= \mathrm{tr}_{\sigma^{-1}}\big[X[P]_N^{(0)}(1),\dots,X[P]_N^{(0)}(\ell_P),*\,\big], \\
\tau(\sigma;P) &:= \tau_{\sigma^{-1}}\big[X[P]^{(0)}(1),\dots,X[P]^{(0)}(\ell_P),*\,\big], \\
\mathcal{E}(P;\tau_{\Xi^\mathrm{lib}(N)}) &:= \sum_{\sigma \in \mathfrak{S}_{\ell_P+1}} c_\sigma(P)\,\tau_{\Xi^\mathrm{lib}(N)}(\sigma;P), \\
\mathcal{E}(P;\tau) &:= \sum_{\sigma \in \mathfrak{S}_{\ell_P+1}} c_\sigma(P)\,\tau(\sigma;P).   
\end{align*} 

We are now finalizing the proof by using what we have prepared so far. Let $P_1,\dots,P_m$ be any monomials such as the above $P$, that is, all the letters appearing in those are supported in a finite time interval $[0,T]$, and rewrite $\ell_k := \ell_{P_k} = L(X[P_k])$ and set $L := \max\{ L(P_k) \mid 1 \leq k \leq m \}$ and $C_0 := \max\{ C_{P_k} \mid 1 \leq k \leq m \}$ for simplicity. We have
\begin{equation}\label{Eq17}
\begin{aligned}
&\Bigg|\mathrm{tr}_N\Big(\mathcal{E}(P_1;\tau_{\Xi^\mathrm{lib}(N)})\cdots \mathcal{E}(P_m;\tau_{\Xi^\mathrm{lib}(N)})\Big) -\tau\Big(\mathcal{E}(P_1;\tau)\cdots\mathcal{E}(P_m;\tau)\Big)\Bigg| \\
&\leq 
\sum_{\substack{\sigma_k \in \mathfrak{S}_{\ell_k +1} \\ (1 \leq k \leq m)}} \big|c_{\sigma_1}(P_1)\cdots c_{\sigma_m}(P_m)\big| \times \\
&\qquad\qquad\quad
\Bigg|\mathrm{tr}_N\big(\tau_{\Xi^\mathrm{lib}(N)}(\sigma_1;P_1)\cdots\tau_{\Xi^\mathrm{lib}(N)}(\sigma_m;P_m)\big) - \tau\big(\tau(\sigma_1;P_1)\cdots\tau(\sigma_m;P_m)\big)\Bigg| \\
&\leq 
C_0^m \sum_{\substack{\sigma_k \in \mathfrak{S}_{\ell_k +1} \\ (1 \leq k \leq m)}}  
\Bigg|\mathrm{tr}_N\big(\tau_{\Xi^\mathrm{lib}(N)}(\sigma_1;P_1)\cdots\tau_{\Xi^\mathrm{lib}(N)}(\sigma_m;P_m)\big) - \tau\big(\tau(\sigma_1;P_1)\cdots\tau(\sigma_m;P_m)\big)\Bigg|
\end{aligned}
\end{equation}
by \eqref{Eq15}. Let $\sigma_k = \sigma_k^{(1)}\cdots\sigma_k^{(\sharp(\sigma_k))}$ be the cycle decomposition such that the rightmost cycle $\sigma_k^{(\sharp(\sigma_k))}$ contains $\ell_k+1$. Then we may and do write
\begin{align*}
\tau_{\Xi^\mathrm{lib}(N)}(\sigma_k;P_k)
&= 
\tau_{\Xi^\mathrm{lib}(N)}(Q_k^{(1)})\,\cdots\,\tau_{\Xi^\mathrm{lib}(N)}(Q_k^{(\sharp(\sigma_k)-1)} )\,Q_{k,N}^{(\sharp(\sigma_k))}, \\
\tau(\sigma_k;P_k) 
&= 
\tau(Q_k^{(1)})\,\cdots\,\tau(Q_k^{(\sharp(\sigma_k)-1)})\,Q_k^{(\sharp(\sigma_k))}
\end{align*}
with some monomials $Q_k^{(*)}$ in the $x_{ij}(t)$, $i \neq n+1$, $0 \leq t \leq T$, and the $x_{n+1\,j}(t)$, whose total length is at most $L(P_k) \leq L$ by construction, possibly with $Q_k^{(\sharp(\sigma_k))} = 1$, where $Q_{k,N}^{(\sharp(\sigma_k))}$ denotes the substitution of $\xi_{ij}^\mathrm{lib}(N)(t)$ for $x_{ij}(t)$ into $Q_k^{(\sharp(\sigma_k))}$. It follows that 
\begin{equation}\label{Eq18}
\begin{aligned}
\Bigg|\mathrm{tr}_N&\big(\tau_{\Xi^\mathrm{lib}(N)}(\sigma_1;P_1)\cdots\tau_{\Xi^\mathrm{lib}(N)}(\sigma_m;P_m)\big) - \tau\big(\tau(\sigma_1;P_1)\cdots\tau(\sigma_m;P_m)\big)\Bigg| \\
&\leq 
\Bigg|\Big(\prod_{k=1}^m \tau_{\Xi^\mathrm{lib}(N)}(Q_k^{(1)})\cdots\tau_{\Xi^\mathrm{lib}(N)}(Q_k^{(\sharp(\sigma_k)-1)}) \Big)\tau_{\Xi^\mathrm{lib}(N)}(Q_1^{(\sharp(\sigma_1))}\cdots Q_m^{(\sharp(\sigma_m))}) \\
&\qquad\qquad\qquad\qquad\qquad- 
\Big(\prod_{k=1}^m \tau(Q_k^{(1)})\cdots\tau(Q_k^{(\sharp(\sigma_k)-1)}) \Big) \tau(Q_1^{(\sharp(\sigma_1))}\cdots Q_m^{(\sharp(\sigma_m))}) \Bigg| \\
&\leq \big(1+\sum_{k=1}^m (\sharp(\sigma_k)-1)\big)  \times (R\vee1)^{Lm} \times 2^{T+1}(2(R\vee1))^{Lm} d\big(\tau_{\Xi^\mathrm{lib}(N)},\tau\big) \\
&\leq (9L^2 m+1) \cdot (R\vee1)^{Lm}\cdot 2^{T+1}(2(R\vee1))^{Lm} d\big(\tau_{\Xi^\mathrm{lib}(N)},\tau\big).
\end{aligned}
\end{equation}
Remark that $\Vert \mathbb{E}_W[X[P_k]_N] \Vert_{M_N(\mathbb{C})} \leq \Vert X[P_k]_N \Vert_{M_N(\mathbb{C})} \leq (R\vee1)^L$ by construction, since the matricial expectation $\mathbb{E}_W[\,-\,]$  is a unital positive map, see Remark \ref{R4.5}. 
Therefore, \eqref{Eq16}--\eqref{Eq18} altogether imply that 
$$
\Bigg| \mathrm{tr}_N\big(\mathbb{E}_W\big[X[P_1]_N\big]\cdots \mathbb{E}_W\big[X[P_m]_N\big]\big) -
\tau\Big(\mathcal{E}(P_1;\tau)\cdots\mathcal{E}(P_m;\tau)\Big) \Bigg| 
\leq 
\frac{C_1}{N^2} + C_2\,d\big(\tau_{\Xi^\mathrm{lib}(N)},\tau\big)  
$$
with constants $C_1, C_2 > 0$ that are independent of the choice of $s$. Then, what we established in the previous subsection, the estimate obtained just above and
$$
\mathbb{E}_W\big[X[P_k]_N\big] = 
\mathbb{E}\big[P_k(\xi_{\bullet\diamond}^\mathrm{lib}(N)(\cdot),U_N^{(\bullet)}((\cdot)\vee s)U_N^{(\bullet)}(s)^*)\mid\mathcal{F}_s\big] = \mathbb{E}\big[P_{k,N}^{(s)}\mid\mathcal{F}_s\big], \quad 1 \leq k \leq m
$$
altogether force $\tau\Big(\mathcal{E}(P_1;\tau)\cdots\mathcal{E}(P_m;\tau)\Big)$ to be $\tau\big(E_s^\tau\big(P^{\tau^s}_1)\cdots E_s^\tau\big(P^{\tau^s}_m\big)\big)$, and we finally obtain that
$$
\Bigg| \mathrm{tr}_N\big(\mathbb{E}\big[P_{1,N}^{(s)}\mid\mathcal{F}_s\big]\cdots\mathbb{E}\big[P_{m,N}^{(s)}\mid\mathcal{F}_s\big]\big) 
-
\tau\big(E_s^\tau\big(P^{\tau^s}_1)\cdots E_s^\tau\big(P^{\tau^s}_m\big)\big) \Bigg| 
\leq 
\frac{C_1}{N^2} + C_2\,d\big(\tau_{\Xi^\mathrm{lib}(N)},\tau\big).  
$$
Since the right-hand side is independent of the choice of $s$, the desired uniform (in time $s$) convergence follows. \qed

\subsection{A slight generalization of \cite[Proposition 3.5]{Levy:arXiv:1112.2452v2}} 
Let $w = w(1)\cdots w(r)$ be a word in the letters $d_1,\dots,d_r$ and $u_1^{\pm1},\dots,u_r^{\pm1}$. Define 
$$
\varepsilon_k := 
\begin{cases} 
+1 & \text{($w(k) = d_*$ or $w(k) = u_* = u_*^{+1}$)}, \\
-1 & \text{($w(k) = u_*^{-1}$)}
\end{cases}
$$
for $1 \leq k \leq r$. In what follows, we may regard $k \mapsto w(k)$ as a function from $\{1,\dots,r\}$ to the letters $d_i, u_i^{\pm1}$. Let $U_N^{(i)}$, $i=1,2,\dots$, be independent left unitary Brownian motions as before, and $D_i \in M_N(\mathbb{C})$, $i=1,2,\dots$, be given matrices. The substitution of $D_i$ and $U_N^{(i)}(t_i)$ for $d_i$ and $u_i$, respectively, into $w$ is denoted by 
$$
w(D_\bullet, U_N^{(\bullet)}(t_\bullet)) = W_N = W_N(1)\cdots W_N(r)
$$
(whose values are taken in $M_N(\mathbb{C})$) with $W_N(k) = D_i$ or $W_N(k) = U_N^{(i)}(t_i)^{\pm1}$. Moreover, we set 
$$
w_\otimes(D_\bullet, U_N^{(\bullet)}(t_\bullet)) = W^\otimes_N = W_N(1)\otimes\cdots\otimes W_N(r) 
$$
(whose values are taken in $M_N(\mathbb{C})^{\otimes r}$).

With the permutation representation $\rho : \mathbb{C}[\mathfrak{S}_r] \curvearrowright (\mathbb{C}^N)^{\otimes r}$ (see subsection 4.2.2) we write 
$$
p_N(t;\sigma) := \mathbb{E}\Big[\frac{1}{N^{\sharp(\sigma)}}\mathrm{Tr}^{\otimes r}(\rho(\sigma) W_N^\otimes)\Big] 
$$
with $t = (t_1,\dots,t_r)$. ({\it n.b.}~ $\sharp(\sigma)$ denotes the number of cycles in $\sigma$ as before.) The family $p_N(t;\sigma)$, $\sigma \in \mathfrak{S}_r$, forms an $r!$ dimension column vector $p_N(t)$ with indices $\mathfrak{S}_r$. We introduce the operation $\Pi_{l,m}^{\varepsilon,\delta}$ on $\mathfrak{S}_r$, $1 \leq l, m \leq r$, $\varepsilon,\delta \in \{\pm1\}$, defined by
$$
\Pi_{l,m}^{\varepsilon,\delta}(\sigma) := 
\begin{cases} 
\sigma(l,m) & (\varepsilon = \delta = +1), \\
(l,m)\sigma & (\varepsilon = \delta = -1), \\
(\sigma(l),m)\sigma & (\varepsilon = +1, \delta = -1), \\
(\sigma(m),l)\sigma = \sigma (\sigma^{-1}(l),m) & (\varepsilon = -1, \delta = +1). 
\end{cases}
$$
A tedious calculation confirms that 
\begin{equation}\label{Eq19} 
\Pi_{l,m}^{\varepsilon,\delta}\circ\Pi_{l',m'}^{\varepsilon',\delta'} = \Pi_{l',m'}^{\varepsilon',\delta'}\circ\Pi_{l,m}^{\varepsilon,\delta} \quad \text{as long as $\{l,m\} \cap \{l',m'\} = \emptyset$} 
\end{equation}
for any choice of $\varepsilon,\varepsilon',\delta,\delta'$.  We also define the $r!\times r!$ matrices $A_i(w)$ (with indices $\mathfrak{S}_r$) by setting the $(\sigma,\sigma')$-entry as
$$
A_i(w)_{\sigma,\sigma'} 
:= 
- \frac{1}{2}\big|w^{-1}(\{u_i^{\pm}\})\big|\,\delta_{\sigma,\sigma'} 
- \sum_{\substack{l,m \in w^{-1}(\{u_i^{\pm}\}) \\ l < m,\, l \overset{\sigma}{\sim} m}} \varepsilon_l \varepsilon_m\, \delta_{\Pi_{l,m}^{\varepsilon_l,\varepsilon_m}(\sigma),\sigma'}, 
$$ 
where $|w^{-1}(\{u_i^{\pm}\})|$ denotes the number of elements of $w^{-1}(\{u_i^{\pm}\})$ and $l \overset{\sigma}{\sim} m$ means that both $l,m$ are in a common cycle of $\sigma$. Then the matrices $A_i(w)$ mutually commute, since the $w^{-1}(\{u_i^{\pm1}\})$ are disjoint. In what follows, $\Vert\,-\,\Vert_\infty$ means the $\ell_\infty$-norm on the $r!$-dimensional vector space of column vectors. The next proposition is just a slight generalization of \cite[Proposition 3.5]{Levy:arXiv:1112.2452v2}, whose proof is a reorganization of the original one. 

\begin{proposition}\label{P4.6} With $\gamma_r := (1,2,\dots,r) \in \mathfrak{S}_r$ we have 
\begin{align*}
&\Big|\mathbb{E}\big[\mathrm{tr}_N(w(D_\bullet,U_N^{(\bullet)}(t_\bullet)))\big]
- 
\sum_{\sigma \in \mathfrak{S}_r} \big(\exp\big(\sum_{i=1}^r t_i A_i(w)\big)\big)_{\gamma_r^{-1},\sigma}\,\frac{1}{N^{\sharp(\sigma)}}\mathrm{Tr}_N^{\otimes r}(\rho(\sigma) w_\otimes(D_i, I_N))\Big| \\ 
&\leq 
\big\Vert p_N(t) - \exp\big(\sum_{i=1}^r t_i A_i(w)\big)p_N(0) \big\Vert_\infty \\
&\leq 
\frac{1}{2N^2}\Big(\sum_{i=1}^r t_i \big|w^{-1}(\{u_i^{\pm1}\})\big|^2\Big) 
\mathrm{e}^{\sum_{i=1}^r  t_i \big|w^{-1}(\{u_i^{\pm1}\})\big|^2} \, \big\Vert p_N(0) \big\Vert_\infty 
\leq 
\frac{r^3 T}{2N^2}\,\mathrm{e}^{r^3 T} \, \big\Vert p_N(0) \big\Vert_\infty
\end{align*}
with $0 = (0,\dots,0)$ and $T := \max_{1\leq i\leq r} t_i$, and furthermore
$$
\Big| \big(\exp\big(\sum_{i=1}^r t_i A_i(w)\big)\big)_{\sigma,\sigma'}\Big| 
\leq \mathrm{e}^{\frac{1}{2}\sum_{i=1}^r t_i \big|w^{-1}(\{u_i^{\pm1}\})\big|^2} 
\leq \mathrm{e}^{\frac{r^3 T}{2}}.  
$$
\end{proposition}  

Remark that $\frac{1}{N^{\sharp(\sigma)}}\mathrm{Tr}_N^{\otimes r}(\rho(\sigma) w_\otimes(D_\bullet, I_N))$ is a product of moments in the $D_i$ with respect to $\mathrm{tr}_N$ of degree less than $r$. Hence \emph{the above proposition (together with the method in the previous subsection) strengthens Biane's asymptotic freeness result \cite[Theorem 1(2)]{Biane:Fields97} for left unitary Brownian motions with constant matrices in the fashion that the convergence as $N\to\infty$ is uniform on finite time intervals}.   

\begin{proof} (A reproduction of the proof of \cite[Proposition 3.5]{Levy:arXiv:1112.2452v2}.) 
The algebra $M_N(\mathbb{C})^{\otimes r}$ has $r$ different $M_N(\mathbb{C})$-bimodule structures 
$$
M_N(\mathbb{C}) \overset{\theta_k^{(+1)}}{\curvearrowright} M_N(\mathbb{C})^{\otimes r} \overset{\theta_k^{(-1)}}{\curvearrowleft} M_N(\mathbb{C})
$$
defined by 
\begin{align*}
\theta_k^{(+1)}(X) (Y_1\otimes\cdots\otimes Y_k \otimes \cdots \otimes Y_r) 
&:= Y_1\otimes\cdots\otimes XY_k \otimes \cdots \otimes Y_r, \\
\theta_k^{(-1)}(X) (Y_1\otimes\cdots\otimes Y_k \otimes \cdots \otimes Y_r) 
&:= Y_1\otimes\cdots\otimes Y_k X \otimes \cdots \otimes Y_r 
\end{align*}
for $X \in M_N(\mathbb{C})$ and $Y_1 \otimes \cdots \otimes Y_r \in M_N(\mathbb{C})^{\otimes r}$. The It\^{o} formula enables us to obtain (see \cite[Lemma 3.7]{Levy:arXiv:1112.2452v2}) that 
\begin{equation}\label{Eq20} 
\begin{aligned}
\frac{\partial}{\partial t_i} p_N(t;\sigma) 
&= 
- 
\frac{1}{2}\big|w^{-1}(\{u_i^{\pm}\})\big|\,p_N(t;\sigma) \\
&\qquad+ 
\sum_{\substack{l,m \in w^{-1}(\{u_i^{\pm1}\}) \\ l < m}} \varepsilon_l \varepsilon_m\,\frac{1}{N^{\sharp(\sigma)}} \mathbb{E}\Big[\mathrm{Tr}_N\big((\theta_l^{-\varepsilon_l}\otimes\theta_m^{-\varepsilon_m})(C_{\mathfrak{u}(N)}) \rho(\sigma) W_N^\otimes\big)\Big], 
\end{aligned}
\end{equation}
where $C_{\mathfrak{u}(N)} = -\frac{1}{N} \sum_{\alpha,\beta=1}^N E_{\alpha\beta}\otimes E_{\beta\alpha}$ with matrix units $E_{\alpha\beta}$ for $M_N(\mathbb{C})$.  
Then, by \cite[Lemmas 3.8 and 3.9]{Levy:arXiv:1112.2452v2} we have 
$$
\frac{1}{N^{\sharp(\sigma)}} \mathbb{E}\Big[\mathrm{Tr}_N\big((\theta_l^{-\varepsilon_l}\otimes\theta_m^{-\varepsilon_m})(C_{\mathfrak{u}(N)}) \rho(\sigma) W_N^\otimes\big)\Big] = 
\begin{cases} 
-p_N(t;\Pi_{l,m}^{\varepsilon_l,\varepsilon_m}(\sigma)) & (l \overset{\sigma}{\sim} m), \\
-\frac{1}{N^2}p_N(t;\Pi_{l,m}^{\varepsilon_l,\varepsilon_m}(\sigma)) & (l \overset{\sigma}{\not\sim} m). 
\end{cases}
$$
Therefore, with the $r!\times r!$ matrices $C_i(w)$ (with indices $\mathfrak{S}_r$): 
$$
C_i(w)_{\sigma,\sigma'} 
:= 
- \sum_{\substack{l,m \in w^{-1}(\{u_i^{\pm}\}) \\ l < m,\, l \overset{\sigma}{\not\sim} m}} \varepsilon_l \varepsilon_m\, \delta_{\Pi_{l,m}^{\varepsilon_l,\varepsilon_m}(\sigma),\sigma'},  
$$ 
we can rewrite \eqref{Eq20} as   
$$
\frac{\partial}{\partial t_i} p_N(t) = \big(A_i(w) + \frac{1}{N^2} C_i(w)\big)p_N(t) \quad (i=1,\dots,r), 
$$ 
which implies that 
$$
p_N(t) = \exp\Big(\sum_i t_i A_i(w) + \frac{1}{N^2} \sum_i t_i C_i(w)\Big)p_N(0),  
$$ 
since the $A_i(w)$ and the $C_i(w)$ mutually commute due to \eqref{Eq19}. 

Let $\Vert\,-\,\Vert$ denote the operator norm with respect to $\Vert\,-\,\Vert_\infty$ on the $r!$-dimensional vector space of column vectors. Observe that 
\begin{align*}
\Vert C_i(w)\Vert 
&\leq 
\frac{\big|w^{-1}(\{u_i^{\pm1}\})\big|\big(\big|w^{-1}(\{u_i^{\pm1}\})\big|-1\big)}{2} \leq \frac{1}{2}\big|w^{-1}(\{u_i^{\pm1}\})\big|^2, \\
\Vert A_i(w)\Vert 
&\leq 
\frac{1}{2}\big|w^{-1}(\{u_i^{\pm1}\})\big|^2. 
\end{align*}
Write $A :=  \sum_i t_i A_i(w)$ and $C :=  \sum_i t_i C_i(w)$ for simplicity. Then we have 
\begin{align*}
\Big\Vert &p_N(t) - (\exp A)\,p_N(0) \Big\Vert_\infty \\
&\leq 
\Big\Vert \exp\big(A + \frac{1}{N^2}C\big) - \exp A \Big\Vert \Vert p_N(0)\Vert_\infty \\
&= 
\Big\Vert \int_0^1 \frac{\mathrm{d}}{\mathrm{d}s}\big(\exp\big(s(A+ \frac{1}{N^2}C\big)\big)\exp((1-s)A)\big)\,\mathrm{d}s\Big\Vert\,\Vert p_N(0)\Vert_\infty\\
&\leq 
\Bigg(\int_0^1 \Big\Vert \exp\big(s\big(A+ \frac{1}{N^2}C\big)\big) \Big\Vert\,\Big\Vert\frac{1}{N^2}C\Big\Vert\,\Big\Vert \exp((1-s)A) \Big\Vert\,\mathrm{d}s\Bigg)\,\Vert p_N(0)\Vert_\infty \\
&\leq
\Bigg(\frac{1}{N^2} \Vert C\Vert\,\mathrm{e}^{\Vert A\Vert}  \int_0^1\mathrm{e}^{\frac{s}{N^2}\Vert C\Vert}\,\mathrm{d}s\Bigg)\,\Vert p_N(0)\Vert_\infty \\
&\leq 
\frac{1}{N^2} \Vert C\Vert\,\mathrm{e}^{\Vert A\Vert + \Vert C\Vert}\,\Vert p_N(0)\Vert_\infty \\
&\leq 
\frac{1}{2N^2}\Big(\sum_{i=1}^r t_i \big|w^{-1}(\{u_i^{\pm1}\})\big|^2\Big) 
\exp\big(\sum_{i=1}^r  t_i \big|w^{-1}(\{u_i^{\pm1}\})\big|^2\big)\,\Vert p_N(0)\Vert_\infty. 
\end{align*}  
Hence we are done. \end{proof}

\section{Large Deviation Upper Bound} 

This section is concerned with the proof of the desired large deviation upper bound for $\tau_{\Xi^\mathrm{lib}(N)}$. To this end, we prove in subsection 5.1  the exponential tightness of the sequence of probability measures $\mathbb{P}(\tau_{\Xi^\mathrm{lib}(N)} \in\,\cdot\,)$, and then, in subsection 5.2, introduce and investigate an appropriate rate function by looking at Proposition \ref{P3.2}. In subsection 5.3, with these preparations, we finalize the proof by using Theorem \ref{T4.1} (with Proposition \ref{P2.3}). 

\subsection{Exponential tightness} Let us start with the next exponential estimate for left unitary Brownian motions. This lemma is inspired by the proof of \cite[Lemma 2.5]{CabanalDuvillardGuionnet:AnnProbab01}. 

\begin{proposition}\label{P5.1} Let $U_N$ be an $N\times N$ left unitary Brownian motion as in the introduction. Then 
\begin{align*} 
\mathbb{P}\Big( \sup_{s \leq t \leq s+\delta} \Vert U_N(s) - U_N(t)\Vert_{\mathrm{tr}_N,2} \geq \varepsilon \Big) \leq 2\sqrt{2}\,\mathrm{e}^{-N^2 L (\varepsilon^2-(8L+1)\delta)}
\end{align*}
holds for every $s \geq 0$, $\varepsilon > 0$, $\delta > 0$ and $L > 0$.
\end{proposition} 
\begin{proof} 
With $Z_N(t) := \mathrm{tr}_N(2\mathrm{Re}(I_N - U_N(t)))$ we observe that 
\begin{align*} 
\mathbb{P}\Big(\sup_{s \leq t \leq s+\delta} \Vert U_N(s) - U_N(t)\Vert_{\mathrm{tr}_N,2} \geq \varepsilon\Big) 
&= 
\mathbb{P}\Big(\sup_{0 \leq t \leq \delta} \Vert U_N(s) - U_N(s+t)\Vert_{\mathrm{tr}_N,2}^2 \geq \varepsilon^2\Big) \\
&= 
\mathbb{P}\Big(\sup_{0 \leq t \leq \delta} \Vert U_N(s+t)U_N(s)^* - I_N\Vert_{\mathrm{tr}_N,2}^2 \geq \varepsilon^2\Big) \\
&= 
\mathbb{P}\Big(\sup_{0 \leq t \leq \delta} \Vert U_N(t) - I_N \Vert_{\mathrm{tr}_N,2}^2 \geq \varepsilon^2\Big) \\
&= 
\mathbb{P}\Big(\sup_{0 \leq t \leq \delta}Z_N(t) \geq \varepsilon^2\Big)
\end{align*}
by the left increment property of left unitary Brownian motions. Thus it suffices to estimate $\mathbb{P}(\sup_{0 \leq t \leq \delta}Z_N(t) \geq \varepsilon^2)$ from the above. 

One has
\begin{align*}
2\mathrm{Re}(I_N - U_N(t)) 
&= -\int_0^t \mathrm{i}\,(\mathrm{d}H_N(s)U_N(s) - U_N(s)^*\,\mathrm{d}H_N(s)) + \int_0^t \mathrm{Re}(U_N(s))\,\mathrm{d}s, 
\end{align*}
since $\mathrm{d}U_N(t) = \mathrm{i}\,\mathrm{d}H_N(t)\,U_N(t) - \frac{1}{2} U_N(t)\,\mathrm{d}t$ with $N\times N$ self-adjoint Brownian motion $H_N$. Set 
$$
\widetilde{M}_N(t) := -\int_0^t \mathrm{i}\,(\mathrm{d}H(s)U_N(s) - U_N(s)^*\,\mathrm{d}H(s)) 
= 2\,\mathrm{Re}(I_N - U_N(t)) - \int_0^t \mathrm{Re}(U_N(s))\,\mathrm{d}s,
$$
and observe that $M_N(t) := \mathrm{tr}_N(\widetilde{M}_N(t)) = Z_N(t) - \int_0^t \mathrm{Re}(\mathrm{tr}_N(U_N(s)))\,\mathrm{d}s$ defines a martingale. Let $C_{\alpha\beta}$ be the standard orthogonal basis of $M_N(\mathbb{C})^\mathrm{sa}$ as in the introduction. Then $H_N(t) = \sum_{\alpha,\beta=1}^N \frac{B_{\alpha\beta}(t)}{\sqrt{N}} C_{\alpha\beta}$ with an $N^2$-dimensional standard Brownian motion $B_{\alpha\beta}$. This expression enables us to compute the quadratic variation 
\begin{align*}  
\langle M_N \rangle(t)
&= 
\frac{1}{N^3} \sum_{\alpha,\beta=1}^N \int_0^t \mathrm{Tr}_N\big(\,\mathrm{i}\,(C_{\alpha\beta}U_N(s) - U_N(s)^* C_{\alpha\beta})\big)^2\,\mathrm{d}t \\
&= 
\frac{1}{N^3} \sum_{\alpha,\beta=1}^N \int_0^t \mathrm{Tr}_N\big(\,\mathrm{i}\,(U_N(s) - U_N(s)^*) C_{\alpha\beta}\big)^2\,\mathrm{d}t \\
&= 
\frac{1}{N^2} \int_0^t \big\Vert \mathrm{i}(U_N(s) - U_N(s)^*) \big\Vert_{\mathrm{tr}_N,2}^2\,\mathrm{d}t \leq \frac{4t}{N^2} 
\end{align*}
as in section 3. 

Note that $Z_N(t) = M_N(t) + \int_0^t \mathrm{Re}(\mathrm{tr}_N(U_N(s)))\,\mathrm{d}s \leq |M_N(t)| + t$. Hence, if $\sup_{0 \leq t \leq \delta}Z_N(t) \geq \varepsilon^2$, then we have both $\sup_{0 \leq t \leq \delta}|M_N(t)|  \geq \varepsilon^2 - \delta$ and 
\begin{align*}
&\sup_{0\leq t \leq \delta} \exp(-N^2 L M_N(t)) + \sup_{0\leq t \leq \delta} \exp(N^2 L M_N(t)) \\
&\qquad\geq 
\sup_{0 \leq t \leq \delta} \Big(\exp(-N^2 L M_N(t)) + \exp(N^2 L M_N(t))\Big) \\
&\qquad\geq 
\sup_{0 \leq t \leq \delta} \exp(N^2 L |M_N(t)|) 
\geq 
\mathrm{e}^{N^2 L (\varepsilon^2 - \delta)}
\end{align*} 
for any fixed $L > 0$. Thus we get 
\begin{align*}
&\mathbb{P}\Big(\sup_{0 \leq t \leq \delta} Z_N(t) \geq \varepsilon^2\Big) 
\leq 
\mathbb{P}\Big( \sup_{0 \leq t \leq \delta}|M_N(t)|  \geq \varepsilon^2 - \delta\Big) \\
&\leq 
\mathrm{e}^{-N^2 L (\varepsilon^2 - \delta)} \Bigg\{\mathbb{E}\Big[\sup_{0\leq t \leq \delta} \exp(-N^2 L M_N(t))\Big] + \mathbb{E}\Big[\sup_{0\leq t \leq \delta} \exp(N^2 L M_N(t))\Big]\Bigg\}
\end{align*} 
by Chebyshev's inequality. We have 
\begin{align*} 
&\mathbb{E}\Big[\sup_{0 \leq t \leq \delta}\exp(\pm N^2 L M_N(t))\Big] \\
&= 
\mathbb{E}\Big[\sup_{0 \leq t \leq \delta}\Big(\exp\Big(\pm N^2 L M_N(t) - \frac{1}{2}\langle \pm N^2 L M_N\rangle(t)\Big)\exp\Big(\frac{1}{2}\langle \pm N^2 L M_N\rangle(t)\Big)\Big)\Big] \\
&\leq 
\mathbb{E}\Big[\sup_{0 \leq t \leq \delta}\exp\Big(\pm N^2 L M_N(t) - \frac{1}{2}\langle \pm N^2 L M_N\rangle(t)\Big) \times \exp\Big(\frac{1}{2}\langle \pm N^2 L M_N\rangle(\delta)\Big)\Big] \\
&\leq 
\mathbb{E}\Big[\sup_{0 \leq t \leq \delta}\exp\Big(\pm N^2 L M_N(t) - \frac{1}{2}\langle \pm N^2 L M_N\rangle(t)\Big)^2\Big]^{1/2} \mathbb{E}\Big[\exp\Big(\frac{1}{2}\langle \pm N^2 L M_N\rangle(\delta)\Big)^2\Big]^{1/2} \\
&\leq 
\mathbb{E}\Big[\sup_{0 \leq t \leq \delta}\exp\Big(\pm N^2 L M_N(t) - \frac{1}{2}\langle \pm N^2 L M_N\rangle(t)\Big)^2\Big]^{1/2} \mathbb{E}\Big[\exp\Big(N^4 L^2 \langle M_N\rangle(\delta)\Big)\Big]^{1/2}
\end{align*}
by the Cauchy--Schwarz inequality. Since $t \mapsto \exp\big(\pm N^2 L M_N(t) - \frac{1}{2}\langle \pm N^2 L M_N\rangle(t)\big)$ and $t \mapsto \exp\big(\pm 4 N^2 L M_N(t) - \frac{1}{2}\langle \pm 4 N^2 L M_N\rangle(t)\big)$ are martingales thanks to \cite[Corollary 3.5.13]{KaratzasShreve:Book}, Doob's maximal inequality with `$p=2$' (see e.g., \cite[Theorem 1.3.8(iv)]{KaratzasShreve:Book} with the help of Jensen's inequality) shows that 
\begin{align*}
&\mathbb{E}\Big[\sup_{0 \leq t \leq \delta}\exp\Big(\pm N^2 L M_N(t) - \frac{1}{2}\langle \pm N^2 L M_N\rangle(t)\Big)^2\Big] \\
&\leq 
2\,\mathbb{E}\Big[\exp\Big(\pm N^2 L M_N(\delta) - \frac{1}{2}\langle \pm N^2 L M_N\rangle(\delta)\Big)^2\Big] \\
&= 
2\,\mathbb{E}\Big[\exp\Big(\pm 2 N^2 L M_N(\delta) - 4\langle \pm N^2 L M_N\rangle(\delta) + 3\langle \pm N^2 L M_N\rangle(\delta)\Big)\Big] \\
&\leq 
2\,\mathbb{E}\Big[\exp\Big(\pm 2 N^2 L M_N(\delta) - 4\langle \pm N^2 L M_N\rangle(\delta)\Big)^2\Big]^{1/2} 
\mathbb{E}\Big[\exp\Big(3 \langle \pm N^2 L M_N\rangle(\delta)\Big)^2\Big]^{1/2} \\
&= 
2\,\mathbb{E}\Big[\exp\Big(\pm 4 N^2 L M_N(\delta) - \frac{1}{2}\langle \pm 4 N^2 L M_N\rangle(\delta)\Big)\Big]^{1/2} 
\mathbb{E}\Big[\exp\Big(6 \langle \pm N^2 L M_N\rangle(\delta)\Big)\Big]^{1/2} \\
&=  
2\,\mathbb{E}\Big[\exp\Big(6 N^4 L^2\langle M_N\rangle(\delta)\Big)\Big]^{1/2}. 
\end{align*}
Therefore, we have 
\begin{align*}
\mathbb{E}\Big[&\sup_{0 \leq t \leq \delta}\exp(\pm N^2 L M_N(t))\Big]\\
&\leq 
\sqrt{2}\,\mathbb{E}\Big[\exp\Big(6 N^4 L^2 \langle M_N\rangle(\delta)\Big)\Big]^{1/4}\mathbb{E}\Big[\exp\Big(N^4 L^2 \langle M_N\rangle(\delta)\Big)\Big]^{1/2} 
\leq 
\sqrt{2}\,\mathrm{e}^{8 N^2 L^2 \delta}.  
\end{align*}
Hence we get
\begin{align*}
\mathbb{P}\Big(\sup_{0\leq t\leq\delta}Z_N(t) \geq \varepsilon^2\Big) 
&\leq
\mathrm{e}^{-N^2 L (\varepsilon^2-\delta)} \times 2 \times \sqrt{2}\,\mathrm{e}^{8N^2 L^2 \delta} 
=
2\sqrt{2}\,\mathrm{e}^{-N^2 L (\varepsilon^2 - (8L+1)\delta)}
\end{align*}
for every $L > 0$. 
\end{proof}  

\begin{corollary}\label{C5.2} The sequence of probability measures $\mathbb{P}(\tau_{\Xi^\mathrm{lib}(N)} \in\,\cdot\,)$ on $TS^c\big(C^*_R\big\langle x_{\bullet\diamond}(\cdot)\big\rangle\big)$ is exponentially tight. 
\end{corollary} 
\begin{proof} Observe that 
\begin{align*}
&\sup_{\substack{0 \leq s,t \leq k \\ |s-t| \leq \delta}} \max_{\substack{1 \leq j \leq r(i) \\ 1 \leq i \leq n+1}} \tau_{\Xi^\mathrm{lib}(N)}\big((x_{ij}(s) - x_{ij}(t))^2\big)^{1/2} \\
&\leq 
\max_{\substack{0 \leq \ell \leq [k/\delta] \\ 1 \leq j \leq r(i) \\ 1 \leq i \leq n}} \sup_{\substack{\ell\delta \leq s \leq (\ell+1)\delta \\ s \leq t \leq s+\delta}} \tau_{\Xi^\mathrm{lib}(N)}\big((x_{ij}(s) - x_{ij}(t))^2\big)^{1/2}  \\
&\leq 
\max_{\substack{0 \leq \ell \leq [k/\delta] \\ 1 \leq j \leq r(i) \\ 1 \leq i \leq n}} \sup_{\substack{\ell\delta \leq s \leq (\ell+1)\delta \\ s \leq t \leq s+\delta}} \Big(\tau_{\Xi^\mathrm{lib}(N)}\big((x_{ij}(s) - x_{ij}(\ell\delta))^2\big)^{1/2} + \tau_{\Xi^\mathrm{lib}(N)}\big((x_{ij}(\ell\delta) - x_{ij}(t))^2\big)^{1/2} \Big) \\
&\leq 
2\max_{\substack{0 \leq \ell \leq [k/\delta] \\ 1 \leq j \leq r(i) \\ 1 \leq i \leq n}} \sup_{\ell\delta \leq t \leq (\ell+2)\delta} \tau_{\Xi^\mathrm{lib}(N)}\big((x_{ij}(\ell\delta) - x_{ij}(t))^2\big)^{1/2} \\
&\leq 
4R\max_{\substack{0 \leq \ell \leq [k/\delta] \\ 1 \leq i \leq n}} \sup_{\ell\delta \leq t \leq (\ell+2)\delta} \big\Vert U_N^{(i)}(\ell\delta) - U_N^{(i)}(t)\big\Vert_{\mathrm{tr}_N,2},
\end{align*}
where $\ell$ is a parameter of non-negative integers and $[k/\delta]$ denotes the greatest non-negative integer that is not greater than $k/\delta$. 
Hence, for each $k \in \mathbb{N}$ and for any $\delta > 0$ and $L >0$, we have 
\begin{align*} 
&\mathbb{P}\Big(\sup_{\substack{0 \leq s,t \leq k \\ |s-t| \leq \delta}} \max_{\substack{1 \leq j \leq r(i) \\ 1 \leq i \leq n+1}} \tau_{\Xi^\mathrm{lib}(N)}\big((x_{ij}(s) - x_{ij}(t))^2\big)^{1/2} > 1/k \Big) \\
&\leq 
\mathbb{P}\Big(\max_{\substack{0 \leq \ell \leq [k/\delta] \\ 1 \leq i \leq n}} \sup_{\ell\delta \leq t \leq (\ell+2)\delta} \big\Vert U_N^{(i)}(\ell\delta) - U_N^{(i)}(t)\big\Vert_{\mathrm{tr}_N,2} \geq \frac{1}{4Rk}\Big) \\
&\leq 
\sum_{\ell=0}^{[k/\delta]} \sum_{i=1}^n \mathbb{P}\Big(\sup_{\ell\delta \leq t \leq (\ell+2)\delta} \big\Vert U_N^{(i)}(\ell\delta) - U_N^{(i)}(t)\big\Vert_{\mathrm{tr}_N,2} \geq \frac{1}{4Rk}\Big) \\
&\leq 
2\sqrt{2}\,n\,([k/\delta]+1)\,\mathrm{e}^{-N^2 L ((16 R^2 k^2)^{-1} - (8L + 1)2\delta)} 
\end{align*}
by Proposition \ref{P5.1}. Therefore, for a given $C > 0$, letting 
$$
L := 32R^2 k^3 C
\quad \text{and} \quad 
\delta_k := \frac{1}{64 R^2 k^2 (256 R^2 k^3 C + 1)},
$$ 
we obtain the following estimate: 
$$
\mathbb{P}\Big(\sup_{\substack{0 \leq s,t \leq k \\ |s-t| \leq \delta_k}} \max_{\substack{1 \leq j \leq r(i) \\ 1 \leq i \leq n+1}} \tau_{\Xi^\mathrm{lib}(N)}\big((x_{ij}(s) - x_{ij}(t))^2\big)^{1/2} > \frac{1}{k} \Big) 
\leq 
C'\,(k^6\,\mathrm{e}^{-N^2 C\,k/2})\,\mathrm{e}^{-N^2 C\,k/2}, 
$$
where $C'>0$ depends only on $n, R, C$ and is independent of $k, N$. If $C > 12$, then $k^6\,\mathrm{e}^{-N^2 C k/2} \leq e^{-N^2 C/2}$. With the sequence $(\delta_k)_{k\geq1}$ it follows that
\begin{align*}
&\mathbb{P}\big( \tau_{\Xi^\mathrm{lib}(N)} \not\in \Gamma_{(\delta_k)}\big) \\
&\leq \sum_{k=1}^\infty \mathbb{P}\Big(\sup_{\substack{0 \leq s,t \leq k \\ |s-t| \leq \delta_k}} \max_{\substack{1 \leq j \leq r(i) \\ 1 \leq i \leq n+1}} \tau_{\Xi^\mathrm{lib}(N)}\big((x_{ij}(s) - x_{ij}(t))^2\big)^{1/2} > \frac{1}{k} \Big) 
\leq C'\,\frac{\mathrm{e}^{-N^2 C}}{1 - \mathrm{e}^{-N^2 C/2}},  
\end{align*}
implying that $\varlimsup_{N\to\infty}\frac{1}{N^2}\log\mathbb{P}\big( \tau_{\Xi^\mathrm{lib}(N)} \not\in \Gamma_{(\delta_k)}\big)  \leq -C$ whenever $C > 12$. 
This together with Lemma \ref{L2.2}(2) shows the exponential tightness of the measures $\mathbb{P}(\tau_{\Xi^\mathrm{lib}(N)} \in\,\cdot\,)$, since $C > 0$ can arbitrarily be large. 
\end{proof}

\subsection{Rate function} We define a map $I^\mathrm{lib}_{\sigma_0} : TS^c\big(C^*_R\big\langle x_{\bullet\diamond}(\cdot)\big\rangle\big) \to [0,+\infty]$ to be  
\begin{equation}\label{Eq21}
\begin{aligned} 
\sup_{\substack{T \geq 0 \\ P = P^* \in \mathbb{C}\langle x_{\bullet\diamond}(\cdot)\rangle}} 
\Bigg\{ &\tau^T(P) - \sigma_0^\mathrm{lib}(P) 
- \frac{1}{2}\sum_{k=1}^n \int_0^T \big\Vert E_s^\tau\big((\mathfrak{D}_s^{(k)}P)(x_{\bullet\diamond}^{\tau^s}(\cdot),v_\bullet^\tau(\cdot))\big) \big\Vert_{\tau,2}^2\,\mathrm{d}s\Bigg\}.
\end{aligned}
\end{equation}
That the integrand is piece-wisely continuous in $s$ follows from Lemma \ref{L5.5} below together with \eqref{Eq22}: Note that $\mathrm{i}\,\mathfrak{D}_s^{(k)}P$ is self-adjoint if $P = P^*$, and then  
\begin{equation}\label{Eq22}
\begin{aligned}
\big\Vert &E_s^\tau\big((\mathfrak{D}_s^{(k)}P)(x_{\bullet\diamond}^{\tau^s}(\cdot),v_\bullet^\tau(\cdot))\big) \big\Vert_{\tau,2}^2 
= 
\big\Vert E_s^\tau\big((\mathrm{i}\,\mathfrak{D}_s^{(k)}P)(x_{\bullet\diamond}^{\tau^s}(\cdot),v_\bullet^\tau(\cdot))\big) \big\Vert_{\tau,2}^2 \\
&= 
-\tau\Big(E_s^\tau\big((\mathfrak{D}_s^{(k)}P)(v_\bullet^\tau((\cdot -s)\vee0)x_{\bullet\diamond}^\tau(\cdot\wedge s)v_\bullet^\tau((\cdot -s)\vee0)^*,v_\bullet^\tau(\cdot))\big)\big)^2\Big) 
\end{aligned}
\end{equation}
holds for every $P = P^* \in \mathbb{C}\big\langle x_{\bullet\diamond}(\cdot)\big\rangle$. 

\begin{lemma}\label{L5.3} If $I^\mathrm{lib}_{\sigma_0}(\tau) < +\infty$, then $\tau^0 = \sigma_0^\mathrm{lib}$, that is, $\pi_0^*(\tau) = \sigma_0$, and 
$$
I^\mathrm{lib}_{\sigma_0}(\tau) = \frac{1}{2} \sup_{\substack{T \geq 0 \\ P = P^* \in \mathbb{C}\langle x_{\bullet\diamond}(\cdot)\rangle}}  \left\{ 
\frac{(\tau^T(P) - \sigma_0^\mathrm{lib}(P))^2}{\sum_{k=1}^n \int_0^T \big\Vert E_s^\tau\big((\mathfrak{D}_s^{(k)}P)(x_{\bullet\diamond}^{\tau^s}(\cdot),v_\bullet^\tau(\cdot))\big) \big\Vert_{\tau,2}^2\,\mathrm{d}s}\right\} 
$$
holds (and the right-hand side is well-defined with convention $0/0=0$, that is, if the denominator is zero, then the numerator must be zero).  
\end{lemma}
\begin{proof} For each fixed $P = P^* \in \mathbb{C}\big\langle x_{\bullet\diamond}(\cdot)\big\rangle$, let $\alpha_T(P) := \tau^T(P) - \sigma_0^\mathrm{lib}(P)$ and $\beta_T(P) := \sum_{k=1}^n \int_0^T \big\Vert E_s^\tau\big((\mathfrak{D}_s^{(k)}P)(x_{\bullet\diamond}^{\tau^s}(\cdot),v_\bullet^\tau(\cdot))\big) \big\Vert_{\tau,2}^2\,ds$, and consider the function 
$$
f_{P,T}(r) := \alpha_T(rP) - \frac{\beta_T(rP)}{2} = \alpha_T(P)\,r - \frac{\beta_T(P)}{2}\,r^2 = -\frac{\beta_T(P)}{2}\Big(r-\frac{\alpha_T(P)}{\beta_T(P)}\Big)^2 + \frac{\alpha_T(P)^2}{2\beta_T(P)}
$$ 
on the real line. If $\beta_T(P) \gneqq 0$, then $\max_r f_{P,T}(r) = f_{P,T}(\alpha_T(P)/\beta_T(P)) = \alpha_T(P)^2/2\beta_T(P)$; otherwise 
$$
\sup_r f_{P,T}(r) = \sup_r \alpha_T(P) r = \begin{cases} 0 & (\alpha_T(P) = 0), \\ +\infty & (\alpha_T(P) \neq 0). \end{cases}
$$

Trivially $\beta_0(P) = 0$ always holds, and hence the above discussion shows that $\alpha_0(P)$ must be $0$ for every $P$, since $I^\mathrm{lib}_{\sigma_0}(\tau) < +\infty$. Therefore, we have proved the former assertion $\tau^0 = \sigma_0^\mathrm{lib}$. 

For any $\varepsilon > 0$, there exist $P_\varepsilon = P_\varepsilon^* \in \mathbb{C}\big\langle x_{\bullet\diamond}(\cdot)\big\rangle$ and $T_\varepsilon \geq 0$ so that $I^\mathrm{lib}_{\sigma_0}(\tau) - \varepsilon < f_{P_\varepsilon,T_\varepsilon}(1) \leq \max_r f_{P_\varepsilon,T_\varepsilon}(r) \leq I^\mathrm{lib}_{\sigma_0}(\tau) < +\infty$. Then, the first paragraph shows that 
$$
I^\mathrm{lib}_{\sigma_0}(\tau) - \varepsilon < \frac{\alpha_{T_\varepsilon}(P_\varepsilon)^2}{2\beta_{T_\varepsilon}(P_\varepsilon)} \leq \sup_{P,T} \frac{\alpha_T(P)^2}{2\beta_T(P)} = \sup_{P,T}\max_r f_{P,T}(r) \leq I^\mathrm{lib}_{\sigma_0}(\tau)
$$ 
with convention $0/0 = 0$. Hence the latter assertion holds. 
\end{proof}   

Here is a simple lemma. 

\begin{lemma}\label{L5.4} Let $(\mathcal{M},\tau)$ be a tracial $W^*$-probability space with $\tau$ faithful, and $u \in \mathcal{M}$ be a unitary, and $\mathcal{N}$ be a (unital) $W^*$-subalgebra of $\mathcal{M}$. Let $E : \mathcal{M} \to \mathcal{N}$ be the unique $\tau$-preserving conditional expectation. If $u$ is $*$-freely independent of $\mathcal{N}$ we have $E(uxu^*) = \tau(x)1 + |\tau(u)|^2 x^\circ$ for every $x \in \mathcal{N}$ with $x^\circ := x - \tau(x)1$. 
\end{lemma}
\begin{proof} For every $y \in \mathcal{N}$, we have $\tau(uxu^* y) = \tau(x)\tau(y) + |\tau(u)|^2\tau(x^\circ y)$ by the $*$-free independence between $u$ and $\mathcal{N}$. Since $E(uxu^*) \in \mathcal{N}$ is uniquely determined by the relation $\tau(uxu^* y) = \tau(E(uxu^*)y)$ for every $y \in \mathcal{N}$, the desired assertion immediately follows.  
\end{proof} 

The same idea as above shows the next lemma. 

\begin{lemma}\label{L5.5} Let $(\mathcal{M},\tau)$ be a tracial $W^*$-probability space with $\tau$ faithful. Let $\mathcal{L}$ and $\mathcal{N}$ be freely independent (unital) $W^*$-subalgebras of $\mathcal{M}$, and $E : \mathcal{M} \to \mathcal{N}$ be the unique $\tau$-preserving conditional expectation. Then $((a_1,\dots,a_{n-1},a_n),(b_1,\dots,b_{n-1})) \in \mathcal{L}^n \times \mathcal{N}^{n-1} \mapsto  E(a_1 b_1 \cdots a_{n-1} b_{n-1} a_n) \in \mathcal{N}$ is written as a universal polynomial in moments of the $a_i$, moments of the $b_i$ and words in the $b_i$. 
\end{lemma}
\begin{proof} Let us calculate the map
$$
((a_1,\dots,a_{n-1},a_n),(b_1,\dots,b_{n-1},b_n)) \in \mathcal{L}^n \times \mathcal{N}^n \mapsto  \tau(a_1 b_1 \cdots a_{n-1} b_{n-1} a_n b_n). 
$$
By \cite[Proposition 11.4, Theorem 11.16]{NicaSpeicher:Book} $\tau(a_1 b_1 \cdots a_{n-1} b_{n-1} a_n b_n)$ is a universal polynomial in moments of the $a_i$ and moments of the $b_i$. Since the map 
$$
((a_1,\dots,a_{n-1},a_n),(b_1,\dots,b_{n-1},b_n)) \mapsto  \tau(a_1 b_1 \cdots a_{n-1} b_{n-1} a_n b_n)
$$ 
is multilinear, each term of the polynomial includes some joint moments of the $b_i$, where $b_n$ appears only once in a unique joint moment. Then we can obtain the desired assertion in the same way as in the proof of Lemma \ref{L5.4}. 
\end{proof}

We remark that the universal polynomial whose existence we have established admits an explicit formula based on the notation in \cite[Lecture 11]{NicaSpeicher:Book}. 

\medskip
Here is a main result of this subsection. 

\begin{proposition}\label{P5.6} $I^\mathrm{lib}_{\sigma_0} : TS^c\big(C^*_R\big\langle x_{\bullet\diamond}(\cdot)\big\rangle\big) \to [0,+\infty]$ is a good rate function. 
\end{proposition}
\begin{proof} By \eqref{Eq22} together with Lemma \ref{L5.5} we observe that 
$$
\tau \mapsto \big\Vert E_s^\tau\big((\mathfrak{D}_s^{(k)}P)(x_{\bullet\diamond}^{\tau^s}(\cdot),v_\bullet^\tau(\cdot))\big) \big\Vert_{\tau,2}^2
$$
is a continuous function for every $s$. Hence 
$$
\tau \mapsto I_{P,T}(\tau) :=  \tau^T(P) - \sigma_0^\mathrm{lib}(P) 
- \frac{1}{2}\sum_{k=1}^n \int_0^T \big\Vert E_s^\tau\big((\mathfrak{D}_s^{(k)}P)(x_{\bullet\diamond}^{\tau^s}(\cdot),v_\bullet^\tau(\cdot))\big) \big\Vert_{\tau,2}^2\,\mathrm{d}s 
$$
is continuous, and consequently, $I^\mathrm{lib}_{\sigma_0}$ is lower semicontinuous. Therefore, it suffices to prove that the level set $\{I^\mathrm{lib}_{\sigma_0} \leq \lambda\}$ sits in a compact subset for every non-negative real number $\lambda \geq 0$. 

Assume that $I^\mathrm{lib}_{\sigma_0}(\tau) \leq \lambda$. By Lemma \ref{L5.3} we have 
\begin{equation}\label{Eq23}
\tau^T(P) \leq \sigma_0^\mathrm{lib}(P) + \sqrt{2\lambda\sum_{k=1}^n \int_0^T \big\Vert E_s^\tau\big((\mathfrak{D}_s^{(k)}P)(x_{\bullet\diamond}^{\tau^s}(\cdot),v_\bullet^\tau(\cdot))\big) \big\Vert_{\tau,2}^2\,\mathrm{d}s}
\end{equation}
for every $P = P^* \in \mathbb{C}\big\langle x_{\bullet\diamond}(\cdot)\big\rangle$ and $T \geq 0$.

For $0 \leq t_1 < t_2$ we have
\begin{align*}
\mathfrak{D}_s^{(k)}\big((x_{ij}(t_1) - x_{ij}(t_2))^2\big) &= 
2\delta_{k,i}\Big\{\mathbf{1}_{[0,t_1]}(s) v_i(t_1-s)^*[x_{ij}(t_1),x_{ij}(t_2)]v_i(t_1-s) \\
&\qquad\qquad\qquad+ \mathbf{1}_{[0,t_2]}(s)v_i(t_2-s)^*[x_{ij}(t_2),x_{ij}(t_1)]v_i(t_2-s)\Big\},
\end{align*}
and hence
\begin{align*}
&\big(\mathfrak{D}_s^{(k)}((x_{ij}(t_1) - x_{ij}(t_2))^2)\big)(x_{ij}^{\tau^s}(\cdot),v_i^\tau(\cdot)) \\
&= 
\begin{cases}
\begin{aligned} 
&2\delta_{k,i}\Big\{ [x_{ij}^\tau(s), v_i^\tau(t_1-s)^*v_i^\tau(t_2-s)x_{ij}^\tau(s)v_i^\tau(t_2-s)^* v_i^\tau(t_1-s)] \\
&\qquad\qquad 
+[x_{ij}^\tau(s),v_i^\tau(t_2-s)^*v_i^\tau(t_1-s)x_{ij}^\tau(s)v_i^\tau(t_1-s)^* v_i^\tau(t_2-s)]\Big\}
\end{aligned}
&(s \leq t_1), \\
2\delta_{k,i} [x_{ij}^\tau(s), v_i^\tau(t_2-s)^* x_{ij}^\tau(t_1) v_i^\tau(t_2-s)] & (t_1 < s \leq t_2), \\
0 & (t_2 < s). 
\end{cases}
\end{align*}
When $s \leq t_1$, Lemma \ref{L5.4} enables us to compute
\begin{align*}
&E_s^\tau\big(\big(\mathfrak{D}_s^{(k)}((x_{ij}(t_1) - x_{ij}(t_2))^2)\big)(x_{ij}^{\tau^s}(\cdot),v_i^\tau(\cdot))\big) \\
&= 
2\delta_{k,i}\Big\{
\big[x_{ij}^\tau(s),E_s^\tau\big(v_i^\tau(t_1-s)^*v_i^\tau(t_2-s)x_{ij}^\tau(s)v_i^\tau(t_2-s)^* v_i^\tau(t_1-s)\big)\big] \\
&\qquad\qquad+ 
\big[x_{ij}^\tau(s),E_s^\tau\big(v_i^\tau(t_2-s)^*v_i^\tau(t_1-s)x_{ij}^\tau(s)v_i^\tau(t_1-s)^* v_i^\tau(t_2-s)\big)\big] \Big\} \\
&= 
2\delta_{k,i}\Big\{
\big[x_{ij}^\tau(s),\big(\tau(x_{ij}^\tau(s))1 + \big|\tau(v_i^\tau(t_1-s)^*v_i^\tau(t_2-s))\big|^2 x_{ij}^\tau(s)^\circ\big)\big] \\
&\qquad\qquad+ 
\big[x_{ij}^\tau(s),\big(\tau(x_{ij}^\tau(s))1 + \big|\tau(v_i^\tau(t_2-s)^*v_i^\tau(t_1-s))\big|^2 x_{ij}^\tau(s)^\circ\big)\big] \Big\} \\
 &= 0. 
\end{align*}
In this way, we obtain the formula: 
\begin{equation}\label{Eq24}
\begin{aligned}
E_s^\tau\big(&\big(\mathfrak{D}_s^{(k)}((x_{ij}(t_1) - x_{ij}(t_2))^2)\big)(x_{ij}^{\tau^s}(\cdot),v_i^\tau(\cdot))\big) \\
&= 2\delta_{k,i} \mathbf{1}_{(t_1,t_2]}(s) |\tau(v_i^\tau(t_2 - s))|^2 [x_{ij}^\tau(s),x_{ij}^\tau(t_1)^\circ].
\end{aligned}
\end{equation}
Then, \eqref{Eq23} with $P := (x_{ij}(t_1) - x_{ij}(t_2))^2$ and $T$ large enough, and \eqref{Eq24} altogether show that 
$$
\tau((x_{ij}(t_1)-x_{ij}(t_2))^2) 
\leq 
\sigma_0^\mathrm{lib}((x_{ij}(t_1)-x_{ij}(t_2))^2) + 
8R^2 \sqrt{2\lambda|t_1-t_2|}.
$$
By the construction of $\sigma_0^\mathrm{lib}$ (see section 2), we see that $\sigma_0^\mathrm{lib}((x_{n+1\,j}(t_1)-x_{n+1\,j}(t_2))^2) = 0$ and moreover that, if $1 \leq i \leq n$, then 
\begin{align*}
\sigma_0^\mathrm{lib}((x_{ij}(t_1)-x_{ij}(t_2))^2) 
&= 
\Vert v_i(t_1) x_{ij}^{\sigma_0} v_i(t_1)^* - v_i(t_2) x_{ij}^{\sigma_0} v_i(t_2)^* \Vert_{\tilde{\sigma}_0,2}^2 \\
&\leq 
\big(2R \Vert v_i(t_1) - v_i(t_2) \Vert_{\tilde{\sigma}_0,2}\big)^2 \\
&= 
4R^2 \Vert v_i(|t_1 - t_2|) - 1 \Vert_{\tilde{\sigma}_0,2}^2 \to 0  
\end{align*}  
as $|t_1 - t_2| \to 0$. Hence, by Lemma \ref{L2.2}(2), $\{ I^\mathrm{lib}_{\sigma_0} \leq \lambda \}$ sits inside a compact subset. 
\end{proof}

We give a few important properties on the rate function $I^\mathrm{lib}_{\sigma_0}$. 

\begin{proposition}\label{P5.7} For any $\tau \in TS^c\big(C^*_R\big\langle x_{\bullet\diamond}(\cdot)\big\rangle\big)$ we have: 
\begin{itemize}
\item[(1)] $I^\mathrm{lib}_{\sigma_0}(\tau) < +\infty$ implies that $t \mapsto x_{n+1\,j}^\tau(t)$ is a constant process for every $1 \leq j \leq r(n+1)$. 
\item[(2)] $I^\mathrm{lib}_{\sigma_0}(\tau) < +\infty$ implies that for each fixed $1 \leq i \leq n$ and $t \geq 0$, we have 
$\pi_t^*(\tau)(P) = \sigma_0(P)$ for every non-commutative polynomial $P$ in indeterminates $x_{ij}$, $1 \leq j \leq r(i)$.  
\item[(3)] $I^\mathrm{lib}_{\sigma_0}(\tau) =0$ if and only if $\tau = \sigma_0^\mathrm{lib}$. Hence $\sigma_0^\mathrm{lib}$ is a unique minimizer of $I_{\sigma_0}^\mathrm{lib}$. 
\end{itemize}
\end{proposition}
\begin{proof} (1) By \eqref{Eq23} and \eqref{Eq24} we have $\Vert x_{n+1\,j}^\tau(t) - x_{n+1\,j}^\tau(0)\Vert_{\tau,2}^2 \leq \Vert x_{n+1\,j}^{\sigma_0^\mathrm{lib}}(t) - x_{n+1\,j}^{\sigma_0^\mathrm{lib}}(0)\Vert_{\sigma_0^\mathrm{lib},2}^2 = 
\Vert x_{n+1\,j}^{\sigma_0} - x_{n+1\,j}^{\sigma_0}\Vert_{\sigma_0,2}^2 = 0$. Hence $ x_{n+1\,j}^\tau(t) = x_{n+1\,j}^\tau(0)$ holds for every $t \geq 0$. 
 
\medskip
(2) Let $P$ be an arbitrary, non-commutative polynomial in indeterminates $x_{ij}$, $1 \leq j \leq r(i)$, with a fixed $1 \leq i \leq n$. It is easy to see that $\mathfrak{D}^{(k)}_s \pi_t(P) = 0$. Hence we have 
$$
r\big(\pi_t^*(\tau^T)(P) - \pi_t^*(\sigma_0^\mathrm{lib})(P)\big) = \tau^T(r\pi_t(P)) - \sigma_0^\mathrm{lib}(r\pi_t(P)) \leq I^\mathrm{lib}_{\sigma_0}(\tau) < +\infty
$$ 
for every $r \in \mathbb{R}$ and $T \geq 0$, and thus $\pi_t^*(\tau)(P) = \pi_t^*(\tau^T)(P) = \pi_t^*(\sigma_0^\mathrm{lib})(P) = \sigma_0(P)$ with $T$ large enough. 

\medskip
(3) By the left increment property of left free unitary Brownian motions (see \cite[Definition 2]{Biane:Fields97}), it is easy to see that $(\sigma_0^\mathrm{lib})^T = \sigma_0^\mathrm{lib}$ holds for every $T \geq 0$. Thus, we trivially obtain that
$$
I_{\sigma_0}^\mathrm{lib}(\sigma_0^\mathrm{lib}) = \sup_{\substack{T \geq 0 \\ P = P^* \in \mathbb{C}\langle x_{\bullet\diamond}(\cdot) \rangle}}\Bigg\{ - \int_0^T \sum_{k=1}^n \big\Vert E_s^{\sigma_0^\mathrm{lib}}\big((\mathfrak{D}_s^{(k)}P)(x_{\bullet\diamond}^{\tau^s}(\cdot),v_\bullet^\tau(\cdot))\big) \big\Vert_{\sigma_0^\mathrm{lib},2}^2\,\mathrm{d}s \Bigg\} = 0.
$$

Lemma \ref{L5.3} with its proof actually shows that $I_{\sigma_0}^\mathrm{lib}(\tau) = 0$ implies that
$$
0 \leq 
\frac{(\tau^T(P) - \sigma_0^\mathrm{lib}(P))^2}{2\sum_{k=1}^n \int_0^T \big\Vert E_s^\tau\big((\mathfrak{D}_s^{(k)}P)(x_{\bullet\diamond}^{\tau^s}(\cdot),v_\bullet^\tau(\cdot))\big) \big\Vert_{\tau,2}^2\,\mathrm{d}s} 
\leq I_{\sigma_0}^\mathrm{lib}(\tau) = 0
$$
(with convention $0/0 = 0$) for all $P = P^* \in \mathbb{C}\big\langle x_{\bullet\diamond}(\cdot)\big\rangle$ and $T \geq 0$. This (with the proviso in Lemma \ref{L5.3}) actually shows that $\tau^T(P) = \sigma_0^\mathrm{lib}(P)$ holds for every $P = P^* \in \mathbb{C}\big\langle x_{\bullet\diamond}(\cdot)\big\rangle$ and $T \geq 0$. This immediately implies that $\tau = \sigma_0^\mathrm{lib}$. 
\end{proof}

These properties actually show that $I_{\sigma_0}^\mathrm{lib}$ is indeed a `right' rate function for our purpose. Further analysis of this rate function $I_{\sigma_0}^\mathrm{lib}$ will be given in a sequel to this article. 

\subsection{Main results} We are ready to prove the next main result of this article. 

\begin{theorem}\label{T5.8} For every closed subset $\Lambda$ of $TS^c\big(C^*_R\big\langle x_{\bullet\diamond}(\cdot)\big\rangle\big)$ we have 
$$
\varlimsup_{N\to\infty}\frac{1}{N^2}\log\mathbb{P}\big(\tau_{\Xi^\mathrm{lib}(N)} \in \Lambda\big) \leq - \inf\big\{ I^\mathrm{lib}_{\sigma_0}(\tau)\,\big|\,\tau \in \Lambda \big\}.
$$
\end{theorem}
\begin{proof}
Since the $\mathbb{P}(\tau_{\Xi^\mathrm{lib}(N)} \in \cdot)$ form an exponentially tight sequence of probability measures and $I^\mathrm{lib}_{\sigma_0}$ is a good rate function, it suffices to prove the following weak large deviation upper bound: 
$$
\varlimsup_{\varepsilon\searrow0}\varlimsup_{N\to\infty} \frac{1}{N^2}\log\mathbb{P}\big(d\big(\tau_{\Xi^\mathrm{lib}(N)},\tau\big) <\varepsilon\big) \leq - I^\mathrm{lib}_{\sigma_0}(\tau)
$$
for every $\tau \in TS^c\big(C^*_R\big\langle x_{\bullet\diamond}(\cdot)\big\rangle\big)$. (This is a standard fact in large deviation theory; see the proofs of \cite[Theorem 4.1.11, Lemma 1.2.18]{DemboZeitouni:Book}.)

Consider the random variable 
\begin{align*} 
I_{P,T,N} &:= \mathbb{E}\big[\tau_{\Xi^\mathrm{lib}(N)}(P)\mid\mathcal{F}_T\big] - \mathbb{E}\big[\tau_{\Xi^\mathrm{lib}(N)}(P)\big] \\
&\qquad- \frac{1}{2}\sum_{k=1}^n \int_0^T \big\Vert \mathbb{E}\big[(\mathfrak{D}_s^{(k)}P)(\xi_{\bullet\diamond}^\mathrm{lib}(N), U_N^{(\bullet)}(\cdot+s)U_N^{(\bullet)}(s)^*)\mid\mathcal{F}_s\big] \big\Vert_{\mathrm{tr}_N,2}^2\,\mathrm{d}s.
\end{align*}
By Proposition \ref{P3.2} we have
\begin{equation}\label{Eq25}
\mathbb{E}[\exp(N^2 I_{P,T,N})] = \mathbb{E}[\exp(N^2 I_{P,0,N})]  = 1.
\end{equation} 

Let $I_{P,T}(\tau)$ be as in the proof of Proposition \ref{P5.6}. We have
\begin{align*} 
\mathbb{P}\big(d\big(\tau_{\Xi^\mathrm{lib}(N)},\tau\big) < \varepsilon\big) 
&= 
\mathbb{E}\Big[\mathbf{1}_{\big\{d\big(\tau_{\Xi^\mathrm{lib}(N)},\tau\big) < \varepsilon\big\}} \exp(N^2 I_{P,T,N} - N^2 I_{P,T,N})\Big] \\
&\leq 
\mathbb{E}\Big[\mathbf{1}_{\big\{d\big(\tau_{\Xi^\mathrm{lib}(N)},\tau\big) < \varepsilon\big\}} \exp(N^2 I_{P,T,N})\Big] \\
&\qquad\times\mathrm{esssup}\big\{\exp(- N^2 I_{P,T,N})\,\big|\, d\big(\tau_{\Xi^\mathrm{lib}(N)},\tau\big) < \varepsilon\big\} \\
&\leq 
\mathrm{esssup}\big\{\exp(- N^2 I_{P,T,N})\,\big|\, d\big(\tau_{\Xi^\mathrm{lib}(N)},\tau\big) < \varepsilon\big\} \quad \text{(use \eqref{Eq25})} \\
&= 
\exp\Big(-N^2\mathrm{essinf}\big\{ I_{P,T,N}\,\big|\, d\big(\tau_{\Xi^\mathrm{lib}(N)},\tau\big) < \varepsilon\big\}\Big).
\end{align*}
Observe that 
\begin{align*}
I_{P,T,N} 
&\geq I_{P,T}(\tau) - |I_{P,T,N} - I_{P,T}(\tau)| \\
&\geq I_{P,T}(\tau) - \mathrm{esssup}\big\{|I_{P,T,N} - I_{P,T}(\tau)|\,\big|\,d\big(\tau_{\Xi^\mathrm{lib}(N)},\tau\big)<\varepsilon\big\}
\end{align*}
holds almost surely on $\big\{d\big(\tau_{\Xi^\mathrm{lib}(N)},\tau\big)<\varepsilon\big\}$.  Therefore, we conclude that 
$$
\frac{1}{N^2}\log\mathbb{P}\big(d\big(\tau_{\Xi^\mathrm{lib}(N)},\tau\big) < \varepsilon\big) 
\leq 
-I_{P,T}(\tau) + \mathrm{esssup}\big\{|I_{P,T,N} - I_{P,T}(\tau)|\,\big|\,d\big(\tau_{\Xi^\mathrm{lib}(N)},\tau\big)<\varepsilon\big\}.   
$$
Then Proposition \ref{P2.3} and Corollary \ref{C4.2} (together with \eqref{Eq3} and \eqref{Eq22}) show that 
$$
\varlimsup_{\varepsilon\searrow0}\varlimsup_{N\to\infty} \mathrm{esssup}\Big\{ \big|I_{P,T,N}  - I_{P,T}(\tau)\big| \,\Big|\, d\big(\tau_{\Xi^\mathrm{lib}(N)},\tau\big) < \varepsilon \Big\} = 0,  
$$
and hence
$$
\varlimsup_{\varepsilon\searrow0}\varlimsup_{N\to\infty}\frac{1}{N^2}\log\mathbb{P}\big(d\big(\tau_{\Xi^\mathrm{lib}(N)},\tau\big) < \varepsilon\big) 
\leq 
-I_{P,T}(\tau)
$$
for every $P = P^* \in \mathbb{C}\big\langle x_{\bullet\diamond}(\cdot)\big\rangle$ and $T \geq 0$. Hence we are done. 
\end{proof}

Here is a standard application of the above large deviation upper bound and Proposition \ref{P5.7}(3). 

\begin{corollary}\label{C5.9} We have $\lim_{N\to\infty}d\big(\tau_{\Xi^\mathrm{lib}(N)},\sigma_0^\mathrm{lib}\big) = 0$ almost surely. 
\end{corollary}
\begin{proof} Let $\varepsilon > 0$ be arbitrarily chosen. By Proposition \ref{P5.6} and Proposition \ref{P5.7}(3) we observe that $\inf\{I_{\sigma_0}^\mathrm{lib}(\tau)\mid d(\tau,\sigma_0^\mathrm{lib}) \geq \varepsilon\} \gneqq 0$. Then, Theorem \ref{T5.8} implies that 
$$
\varlimsup_{N\to\infty}\frac{1}{N^2}\log\mathbb{P}(d(\tau_{\Xi^\mathrm{lib}(N)},\sigma_0^\mathrm{lib}) \geq \varepsilon) \leq -\inf\{I_{\sigma_0}^\mathrm{lib}(\tau)\mid d(\tau,\sigma_0^\mathrm{lib}) \geq \varepsilon\} \lneqq 0. 
$$
Thus we obtain that $\sum_{N=1}^\infty \mathbb{P}(d(\tau_{\Xi^\mathrm{lib}(N)},\sigma_0^\mathrm{lib}) \geq \varepsilon) < +\infty$. Hence the desired assertion follows by the Borel--Cantelli lemma. 
\end{proof}

\section{Discussions} 

One of the motivations in mind is to provide a common basis for the study of Voiculescu's approach (\cite{Voiculescu:AdvMath99}) and our orbital approach (\cite{HiaiMiyamotoUeda:IJM09},\cite{Ueda:IUMJ14}) to the concept of mutual information in free probability. In fact, the key ingredient of Voiculescu's approach is the liberation process, while the orbital approach involves `orbital microstates' by unitary matrices. Thus, a serious lack was a random matrix counterpart of liberation process, whose candidate we introduced in this article. Here we are not going to any detailed discussions about such a study, but only give some comments on it.

\medskip
We may apply the contraction principle in large deviation theory to our large deviation upper bound obtained in section 5. 

\begin{corollary}\label{C6.1} Let $\nu_{N,T}$ be the marginal probability distribution on $\mathrm{U}(N)$ of the $N\times N$ left unitary Brownian motion at time $T > 0$. Define 
$$
I_{\sigma_0,T}^\mathrm{lib}(\sigma) := \inf\big\{ I_{\sigma_0}^\mathrm{lib}(\tau)\,\big|\, \pi_T^*(\tau) = \sigma \big\}, \qquad \sigma \in TS\big(C^*_R\big\langle x_{\bullet\diamond}\big\rangle\big).  
$$
Then for any closed subset $\Lambda$ of $TS\big(C^*_R\big\langle x_{\bullet\diamond}\big\rangle\big)$ we have 
\begin{align*}
\varlimsup_{N\to\infty}\frac{1}{N^2}\log \nu_{N,T}^{\otimes n}\Big(\big\{ \mathbf{U} \in \mathrm{U}(N)^n\,\big|\, \mathrm{tr}^{\Xi(N)}_\mathbf{U}\in \Lambda\big\} \Big)
\leq 
-\inf\big\{ I_{\sigma_0,T}^\mathrm{lib}(\sigma) \mid \sigma \in \Lambda \big\}.    
\end{align*}
Here, $\mathrm{tr}^{\Xi(N)}_\mathbf{U} \in TS\big(C^*_R\big\langle x_{\bullet\diamond}\big\rangle\big)$ with $\mathbf{U} = (U_i)_{i=1}^n \in \mathrm{U}(N)^n$ is defined by 
$\mathrm{tr}^{\Xi(N)}_\mathbf{U}(P) := \mathrm{tr}_N(\Phi_\mathbf{U}(P))$, $\quad P \in \mathbb{C}\big\langle x_{\bullet\diamond} \big\rangle$, where $\Phi_\mathbf{U} : \mathbb{C}\big\langle x_{\bullet\diamond} \big\rangle \to M_N(\mathbb{C})$ is a unique $*$-homomorphism sending $x_{ij}$ $(1 \leq i \leq n)$ to $U_i\xi_{ij}(N)U_i^*$ and $x_{n+1\,j}$ to $\xi_{n+1\,j}(N)$. 
\end{corollary}

We write 
$$
\chi_\mathrm{orb}^T(\sigma) 
:= 
\lim_{\substack{m\to\infty \\ \delta \searrow 0}}\varlimsup_{N\to\infty}\frac{1}{N^2}\log \nu_{N,T}^{\otimes n}\Big(\big\{ \mathbf{U} \in \mathrm{U}(N)^n\,\big|\, \mathrm{tr}^{\Xi(N)}_\mathbf{U}\in \mathcal{O}_{m,\delta}(\sigma) \big\} \Big),
$$
where $\mathcal{O}_{m,\delta}(\sigma)$, $m \in \mathbb{N}$, $\delta > 0$, denotes the (open) subset of $\sigma' \in TS\big(C^*_R\big\langle x_{\bullet\diamond}\big\rangle\big)$ such that $\big|\sigma'(x_{i_1 j_1} \cdots x_{i_p j_p}) - \sigma(x_{i_1 j_1} \cdots x_{i_p j_p})\big| < \delta$ whenever $1 \leq i_k \leq n+1$, $1 \leq j_k \leq r(i_k)$, $1 \leq k \leq p$ and $1 \leq p \leq m$. 

\medskip
A problem in this direction is to show that $\chi_\mathrm{orb}(\sigma) \leq \varliminf_{T\to+\infty}\chi_\mathrm{orb}^T(\sigma)$ holds, where $\chi_\mathrm{orb}(\sigma)$ denotes the orbital free entropy of the random multi-variables $(x_{ij})_{1 \leq j \leq r(i)}$, $1 \leq i \leq n$, under $\sigma$ (see \cite{HiaiMiyamotoUeda:IJM09},\cite{Ueda:IUMJ14}). If this was the case, then we would obtain that $\chi_\mathrm{orb}(\sigma) = \lim_{T\to+\infty}\chi_\mathrm{orb}^T(\sigma)$ (see below) and $\chi_\mathrm{orb}(\sigma) \leq - \varlimsup_{T\to+\infty} I_{\sigma_0,T}^\mathrm{lib}(\sigma)$. Remark that, if the families $\{x_{ij}\}_{1\leq j \leq r(i)}$, $1 \leq i \leq n$, are freely independent under $\sigma_0$, then it is easy to see that $\pi_T^*(\sigma_0^\mathrm{lib}) = \sigma_0$ for all $T \geq 0$, and hence Proposition \ref{P5.7}(3) shows that $I_{\sigma_0,T}^\mathrm{lib}(\sigma_0) = 0$ for all $T \geq 0$ so that  $\chi_\mathrm{orb}(\sigma_0) = - I_{\sigma_0,T}^\mathrm{lib}(\sigma_0)$ holds as $0 = 0$ for all $T \geq 0$. Thus our conjecture seems plausible. 

\medskip
Here we would like to point out that 
$$
\lim_{T\to+\infty}\lim_{N\to\infty} \frac{1}{N^2}\log\max \left\{ \frac{d\nu_{N,T}}{d\nu_N}(U)\,\Big|\, U \in \mathrm{U}(N) \right\} = 
 \lim_{T\to+\infty}\lim_{N\to\infty} \frac{1}{N^2}\log\frac{d\nu_{N,T}}{d\nu_N}(I_N) = 0
$$
with the Haar probability measure $\nu_N$ on $\mathrm{U}(N)$ follows from the formula obtained precisely by L\'{e}vy and M\"{a}ida \cite[Proposition 4.2; Lemma 4.7; Proposition 5.2]{LevyMaida:ESAIM:Proc15} with the aid of the fact that 
$$
K(k) = \int_0^1 \frac{\mathrm{d}s}{\sqrt{(1-s^2)(1-k^2 s^2)}} = -\frac{1}{2}\log(1-k) + \frac{3}{2}\log2 + o(1) \quad (\text{as $k\nearrow 1$}). 
$$
Thus, for any Borel subset $\Lambda$ of $TS\big(C^*_R\big\langle x_{\bullet\diamond}\big\rangle\big)$ we have 
\begin{align*}
&\frac{1}{N^2}\log \nu_{N,T}^{\otimes n}\Big(\big\{ \mathbf{U} \in \mathrm{U}(N)^n\,\big|\, \mathrm{tr}^{\Xi(N)}_\mathbf{U}\in \Lambda \big\} \Big) \\
&\leq  
\frac{1}{N^2}\log \nu_N^{\otimes n}\Big(\big\{ \mathbf{U} \in \mathrm{U}(N)^n\,\big|\, \mathrm{tr}^{\Xi(N)}_\mathbf{U}\in \Lambda \big\} \Big) 
+ \frac{n}{N^2}\log \frac{d\nu_{N,T}}{d\nu_N}(I_N), 
\end{align*} 
implying that $\varlimsup_{T\to\infty}\chi^T_\mathrm{orb}(\sigma) \leq \chi_\mathrm{orb}(\sigma)$ (use \cite[Remark 3.3]{Ueda:Preprint16} at this point). On the other hand, with
$$
L := \varliminf_{T\to+\infty}\varliminf_{N\to\infty} \frac{1}{N^2}\log\min \left\{ \frac{d\nu_{N,T}}{d\nu_N}(U)\,\Big|\, U \in \mathrm{U}(N) \right\} (\leq 0),
$$
a similar consideration as above shows that $\varliminf_{T\to\infty}\chi^T_\mathrm{orb}(\sigma) \geq \chi_\mathrm{orb}(\sigma) + nL$. Hence the problem is whether $L = 0$ or not. We have confirmed this in the affirmative too, and will give a further study on the orbital free entropy in a subsequent paper.

\section*{Acknowledgement} We would like to express our sincere gratitude to the referee for his/her very careful reading of this paper and pointing out a mistake in the original proof of exponential tightness. 

}


\begin{thebibliography}{99}    

\bibitem{AndersonGuionnetZeitouni-Book}
G. Anderson, A. Guionnet and O. Zeitouni, 
{\it An Introduction to Random Matrices.} Cambridge Studies in Advanced Mathematics, 118. Cambridge University Press, 2009. 

\bibitem{Biane:Fields97} Ph. Biane, 
Free Brownian motion, free stochastic calculus and random matrices. {\it Free Probability Theory.} Fields Institute Communications, {\bf 12} (1997), 1--19. 

\bibitem{BianeCapitaineGuionnet:InventMath03} P. Biane, M. Capitaine and A. Guionnet, 
Large deviation bound for matrix Brownian motion. {\it Invent. math.}, {\bf 152} (2003), 433--459. 

\bibitem{CabanalDuvillardGuionnet:AnnProbab01} T. Cabanal Duvillard and A. Guionnet, 
Large deviations upper bounds for the laws of matrix-valued processes and non-commutative entropies. {\it Ann. Probab.}, {\bf 29} (2001), 1205--1261.  

\bibitem{CollinsDahlqvistKemp:PTRF1x} B. Collins, A. Dahlqvist and T. Kemp, 
The spectral edge of unitary Brownian motion. 
{\it Probab. Theory Relat. Fields}, to appear. \url{doi:10.1007/s00440-016-0753-x}  

\bibitem{DemboZeitouni:Book} A. Dembo and O. Zeitouni, 
{\it Large Deviations Techniques and Applications.} 
Springer, 1998.

\bibitem{Guionnet:LNM} A. Guionnet, {\it Large Random Matrices: Lectures on Macroscopic Asymptotics.} 
Lecture Notes in Math., {\bf 1957}, Springer, 2009. 

\bibitem{HiaiMiyamotoUeda:IJM09} F. Hiai, T. Miyamoto and Y. Ueda, 
Orbital approach to microstate free entropy. 
{\it Internat.~J.~Math.}, {\bf 20} (2009), 227--273.

\bibitem{HiaiPetz:Book}
F. Hiai and D. Petz, {\it The Semicircle Law, Free Random Variables and Entropy.} 
Mathematical Surveys and Monographs, Vol. 77, Amer. Math. Soc., 2000.

\bibitem{Hu:Book16} Y. Hu, {\it Analysis on Gaussian Spaces.} 
World Scientific, 2016. 

\bibitem{KaratzasShreve:Book} I. Karatzas and S. Shreve, 
{\it Brownian Motion and Stochastic Calculus.} 
Second edition, GTM {\bf 113}, Springer, 1998. 

\bibitem{Levy:AdvMath08} T. L\'{e}vy, Schur--Wyel duality and the heat kernel measure on the unitary group. {\it Adv. Math.}. {\bf 218} (2008), 537--575. 

\bibitem{Levy:arXiv:1112.2452v2} T. L\'{e}vy, The master field on the plane. 
arXiv:1112.2452v2. 

\bibitem{LevyMaida:ESAIM:Proc15} T. L\'{e}vy and M.~Ma\"{i}da, 
On the Douglas--Kazakov phase transition. 
{\it ESAIM: Proc.}, {\bf 51} (2015), 89--121.  

\bibitem{NicaSpeicher:Book} A. Nica and R.~Speicher, 
{\it Lectures on the Combinatorics of Free Probability.} 
London Mathematical Society Lecture Notes Series, 335. Cambridge University Press, 2006. 

\bibitem{Nualart:Book06} D. Nualart, {\it The Malliavin Calculus and Related Topics.} Second edition. Probability and its Applications. Springer-Verlag, 2006. 

\bibitem{Paulsen:Book} V. Paulsen, 
{\it Completely Bounded Maps and Operator Algebras.} Cambridge Studies in Advanced Mathematics, 78. Cambridge University Press, Cambridge, 2002. 

\bibitem{Rudin:RedBook} W. Rudin, 
{\it Real and Complex Analysis.} 
Third Edition, McGraw-Hill, 1987. 

\bibitem{Ueda:IUMJ14} Y. Ueda, 
Orbital free entropy, revisited. 
{\it Indiana Univ.~Math.~J.}, {\bf 63} (2014), 551--577. 

\bibitem{Ueda:Preprint16} Y. Ueda, 
A remark on orbital free entropy. 
{\it Arch. Math.}, {\bf 108} (2017), 629--638. 

\bibitem{Voiculescu:InventMath91} D. Voiculescu, 
Limit laws for random matrices and free products. 
{\it Invent.~Math.} 104, 1 (1991), 201--220.

\bibitem{Voiculescu:AdvMath99}
D. Voiculescu, The analogue of entropy and of Fisher's information measure in free probability theory VI: liberation and mutual free information. {\it Adv. Math.}, {\bf 146} (1999), 101--166.

\end{thebibliography}
\end{document}